\documentclass[review]{elsarticle}
\usepackage{srcltx}
\usepackage{eurosym}
\usepackage{mathtools}
\usepackage{amsmath}
\usepackage{amsfonts}
\usepackage{amssymb}
\usepackage{amsthm}
\usepackage{graphicx}
\usepackage{mathrsfs}
\usepackage{xcolor}
\usepackage{exscale}
\usepackage{latexsym}
\usepackage[colorlinks,plainpages=true,pdfpagelabels,hypertexnames=true,colorlinks=true,pdfstartview=FitV,linkcolor=blue,citecolor=red,urlcolor=black]{hyperref}
\PassOptionsToPackage{unicode}{hyperref}
\PassOptionsToPackage{naturalnames}{hyperref}
\usepackage{enumerate}
\usepackage[shortlabels]{enumitem}
\usepackage{bookmark}
\usepackage{wasysym}
\usepackage{esint}
\usepackage[ddmmyyyy]{datetime}
\usepackage[margin=2.2cm]{geometry}
\makeatletter
\g@addto@macro\normalsize{%
  \setlength\abovedisplayskip{4pt}
  \setlength\belowdisplayskip{4pt}
  \setlength\abovedisplayshortskip{4pt}
  \setlength\belowdisplayshortskip{4pt}
}
\numberwithin{equation}{section}
\everymath{\displaystyle}
\usepackage[capitalize,nameinlink]{cleveref}
\crefname{section}{Section}{Sections}
\crefname{subsection}{Subsection}{Subsections}
\crefname{condition}{Condition}{Conditions}
\crefname{hypothesis}{Hypothesis}{Conditions}
\crefname{assumption}{Assumption}{Assumptions}
\crefname{lemma}{Lemma}{Lemmas}

\crefformat{equation}{\textup{#2(#1)#3}}
\crefrangeformat{equation}{\textup{#3(#1)#4--#5(#2)#6}}
\crefmultiformat{equation}{\textup{#2(#1)#3}}{ and \textup{#2(#1)#3}}
{, \textup{#2(#1)#3}}{, and \textup{#2(#1)#3}}
\crefrangemultiformat{equation}{\textup{#3(#1)#4--#5(#2)#6}}%
{ and \textup{#3(#1)#4--#5(#2)#6}}{, \textup{#3(#1)#4--#5(#2)#6}}%
{, and \textup{#3(#1)#4--#5(#2)#6}}

\Crefformat{equation}{#2Equation~\textup{(#1)}#3}
\Crefrangeformat{equation}{Equations~\textup{#3(#1)#4--#5(#2)#6}}
\Crefmultiformat{equation}{Equations~\textup{#2(#1)#3}}{ and \textup{#2(#1)#3}}
{, \textup{#2(#1)#3}}{, and \textup{#2(#1)#3}}
\Crefrangemultiformat{equation}{Equations~\textup{#3(#1)#4--#5(#2)#6}}%
{ and \textup{#3(#1)#4--#5(#2)#6}}{, \textup{#3(#1)#4--#5(#2)#6}}%
{, and \textup{#3(#1)#4--#5(#2)#6}}

\crefdefaultlabelformat{#2\textup{#1}#3}
\newtheorem{theorem}{Theorem}[section]
\newtheorem{lemma}[theorem]{Lemma}

\newtheorem{definition}[theorem]{Definition}
\newtheorem{remark}[theorem]{Remark}        

\newtheorem{claim}[theorem]{Claim}
\numberwithin{equation}{section}
\def\aa{\mathcal{A}}

\newcommand{\lamot}{\La_0,\La_1}

\newcommand{\pa}{\partial}

\newcommand{\vlh}{\lsbt{v}{\la,h}}

\newcommand{\elam}{\lsbo{E}{\lambda}}

\newcommand{\mm}{\mathcal{M}}
\newcommand{\tQ}{\tilde{Q}}

\newcommand{\tx}{\tilde{x}}

\newcommand{\tlt}{\tilde{t}}

\newcommand{\tz}{\tilde{z}}

\makeatletter
\newcommand{\vo}{\vec{o}\@ifnextchar{^}{\,}{}}
\makeatother
\def\Yint#1{\mathchoice
    {\YYint\displaystyle\textstyle{#1}}%
    {\YYint\textstyle\scriptstyle{#1}}%
    {\YYint\scriptstyle\scriptscriptstyle{#1}}%
    {\YYint\scriptscriptstyle\scriptscriptstyle{#1}}%
      \!\iint}
\def\YYint#1#2#3{{\setbox0=\hbox{$#1{#2#3}{\iint}$}
    \vcenter{\hbox{$#2#3$}}\kern-.50\wd0}}
\def\longdash{-\mkern-9.5mu-} 
\def\tiltlongdash{\rotatebox[origin=c]{18}{$\longdash$}}
\def\fiint{\Yint\tiltlongdash}
\def\Xint#1{\mathchoice
    {\XXint\displaystyle\textstyle{#1}}%
    {\XXint\textstyle\scriptstyle{#1}}%
    {\XXint\scriptstyle\scriptscriptstyle{#1}}%
    {\XXint\scriptscriptstyle\scriptscriptstyle{#1}}%
      \!\int}
\def\XXint#1#2#3{{\setbox0=\hbox{$#1{#2#3}{\int}$}
    \vcenter{\hbox{$#2#3$}}\kern-.50\wd0}}
\def\hlongdash{-\mkern-13.5mu-}
\def\tilthlongdash{\rotatebox[origin=c]{18}{$\hlongdash$}}
\def\hint{\Xint\tilthlongdash}
\makeatletter
\def\namedlabel#1#2{\begingroup
   \def\@currentlabel{#2}%
   \label{#1}\endgroup
}
\makeatother
\makeatletter
\newcommand{\rmh}[1]{\mathpalette{\raisem@th{#1}}}
\newcommand{\raisem@th}[3]{\hspace*{-1pt}\raisebox{#1}{$#2#3$}}
\makeatother

\newcommand{\lsb}[2]{#1_{\rmh{-3pt}{#2}}}
\newcommand{\lsbo}[2]{#1_{\rmh{-1pt}{#2}}}
\newcommand{\lsbt}[2]{#1_{\rmh{-2pt}{#2}}}
\newcommand{\redref}[2]{\texorpdfstring{\protect\hyperlink{#1}{\textcolor{black}{(}\textcolor{red}{#2}\textcolor{black}{)}}}{}}
\newcommand{\redlabel}[2]{\hypertarget{#1}{\textcolor{black}{(}\textcolor{red}{#2}\textcolor{black}{)}}}

\newcommand{\descitem}[1]{\item[(#1)] \label{#1}}
\newcommand{\descref}[1]{\hyperref[#1]{\textnormal{\textcolor{black}{(}\textcolor{blue}{\bf #1}\textcolor{black}{)}}}}

\newcommand{\dref}[2]{\hyperref[#1]{\textcolor{black}{(}\textcolor{blue}{\bf #2}\textcolor{black}{)}}}

\newcommand{\mfx}{\mathfrak{x}}

\newcommand{\mfz}{\mathfrak{z}}

\newcommand{\mft}{\mathfrak{t}}

\newcommand\RR{\mathbb{R}}

\newcommand\NN{\mathbb{N}}
\newcommand{\al}{\alpha}
\newcommand{\be}{\beta}
\newcommand{\ga}{\gamma}
\newcommand{\de}{\delta}
\newcommand{\ve}{\varepsilon}

\newcommand{\ep}{\epsilon}
\newcommand{\ka}{\kappa}

\newcommand{\om}{\omega}
\newcommand{\la}{\lambda}
\newcommand{\vt}{\vartheta}

\newcommand{\Th}{\Theta}

\newcommand{\Om}{\Omega}

\newcommand{\La}{\Lambda}

\DeclareMathOperator{\dv}{div}

\DeclareMathOperator{\diam}{diam}

\DeclareMathOperator{\loc}{loc}

\newcommand{\iprod}[2]{\langle #1,  #2\rangle}
\newcommand{\norm}[1]{\left|\hspace{-0.2mm}\left| #1 \right|\hspace{-0.2mm}\right|}
\newcommand{\abs}[1]{\left| #1\right|}

\newcommand{\lbr}[1][(]{\left#1}
\newcommand{\rbr}[1][)]{\right#1}

\newcommand{\txt}[1]{\qquad \text{#1} \qquad}

\newcommand{\gflt}{$(\ga,S_0)$-Reifenberg flat }
\newcommand{\aaa}{\overline{\mathcal{A}}}
\newcounter{whitney}
\refstepcounter{whitney}

\newcounter{ineqcounter}
\refstepcounter{ineqcounter}
\makeatletter
\def\ps@pprintTitle{%
\let\@oddhead\@empty
\let\@evenhead\@empty
\def\@oddfoot{}%
\let\@evenfoot\@oddfoot}
\makeatother
\usepackage[titletoc,toc,page]{appendix}

\newcommand{\lv}{\lvert}
\newcommand{\rv}{\rvert}
\newcommand{\lV}{\lVert}
\newcommand{\rV}{\rVert}
\newcommand{\vf}{\vec{f}}

\newcommand{\na}{\nabla}

\newcommand{\vsp}{\vspace{1em}}

\newcommand{\lt}{\left}
\newcommand{\rt}{\right}
\DeclareRobustCommand{\rchi}{{\mathpalette\irchi\relax}}
\newcommand{\irchi}[2]{\raisebox{\depth}{$#1\chi$}} 

\begin{document}

\begin{frontmatter}

\title{An existence result for nonhomogeneous quasilinear parabolic equations beyond the duality pairing}

\author[myaddress,myaddresstwo]{Karthik Adimurthi\corref{mycorrespondingauthor}\tnoteref{thanksfirstauthor}}
\cortext[mycorrespondingauthor]{Corresponding author}
\ead{karthikaditi@gmail.com and kadimurthi@snu.ac.kr}
\tnotetext[thanksfirstauthor]{Supported by the National Research Foundation of Korea grant NRF-2017R1A2B2003877.}

\author[myaddress,myaddresstwo]{Sun-Sig Byun\tnoteref{thankssecondauthor}}
\ead{byun@snu.ac.kr}
\tnotetext[thankssecondauthor]{Supported by the National Research Foundation of Korea grant  NRF-2015R1A4A1041675. }

\author[myaddress]{Wontae Kim\tnoteref{thanksfirstauthor}}
\ead{m20258@snu.ac.kr}

\address[myaddress]{Department of Mathematical Sciences, Seoul National University, Seoul 08826, Korea.}
\address[myaddresstwo]{Research Institute of Mathematics, Seoul National University, Seoul 08826, Korea.}

\begin{abstract}
In this paper, we prove existence of \emph{very weak solutions} to nonhomogeneous quasilinear parabolic equations beyond the duality pairing. The main ingredients are a priori esitmates in suitable weighted spaces combined with the compactness argument developed in \cite{bulicek2018well}. In order to obtain the a priori estimates, we make use of the full Calder\'on-Zygmund machinery developed in the past few years and combine it with some sharp bounds for the subclass of Muckenhoupt weights considered in this paper.  

\end{abstract}


\end{frontmatter}

\section{Introduction}
\label{section1}

In this paper, we are interested in obtaining existence of \emph{very weak solution} to equations of the form
\begin{equation}
    \label{main}
    \left\{\begin{array}{ll}
        u_t - \dv \aa(x,t,\nabla u + \vec{h}) = - \dv |\vec{f}|^{p-2} \vec{f} & \txt{in} \Om \times (0,T),\\
        u = 0 & \txt{on} \pa_p(\Om \times (0,T)),
    \end{array}\right.
\end{equation}
where $\aa(x,t,\zeta)$ is modelled after the well known $p$-Laplace operator, $\Om$ is a bounded domain with possible nonsmooth boundary and $\vec{h}$ is a vector field that satisfies \cref{h_hypothesis}. 

Weak solutions to \eqref{main} are in the space $u \in L^2(0,T; L^2(\Om)) \cap L^p(0,T; W_0^{1,p}(\Om))$ which allows one to use $u$ as a test function. But from the definition of weak solution, we see that the expression (see Definition \ref{very_weak_solution}) makes sense if we only assume $u \in L^2(0,T; L^2(\Om)) \cap L^{s}(0,T; W_0^{1,s}(\Om))$ for some $s > \max\{p-1,1\}$. But under this milder notion of solution called \emph{very weak solution}, we lose the ability to use $u$ as a test function. Subsequently, the technique of parabolic Lipschitz truncation was developed in the seminal paper of \cite{KL} where the authors could study certain properties of \emph{very weak solutions} in the range of $(p-\be,p)$ for all $\be \in (0,\be_0)$ for sufficiently small $\be_0$ depending only on the data. 

One of the main questions was to obtain existence of \emph{very weak solutions} in the range $(p-1,p)$. Though this problem is very formidable, nevertheless, there has been much progress achieved in the past few decades in understanding the question in the range $(p-\be_0,p)$ for all $\be \in (0,\be_0)$. Here $\be_0$ is some small universal constant depending only on the data. \emph{In this section, all the discussion will be restricted to the range $(p-\be_0,p)$ with the choice of $\be_0$ appropriately chosen in the respective papers pertaining to the discussion.}

The quasilinear elliptic analogue of the question considered in this paper was recently solved in \cite{bulivcek2016existence} under suitable assumptions on the nonlinearity $\aa(\cdot,\cdot)$ and the domain $\Om$. In order to prove existence, they made use of a compactness argument developed in \cite{bulivcek2016linearexistence} along with a priori estimates in suitable weighted spaces. The weighted estimates in the linear case were first proved in \cite{MR3318165} and using different techniques, a slightly weaker version was proved in \cite{tadele} and a slightly more weaker version was proved in \cite{bulivcek2016linearexistence}.  The weighted a  priori estimates in the quasilinear case considered in \cite{bulivcek2016existence} were obtained using the ideas that first originated in \cite{Mikkonen}. It should be noted that in \cite{adimurthi2016weight},  the required weighted a priori estimates in the elliptic case were also obtained, albeit with a stronger assumption on the nonlinear structure and a weaker assumption on the domain than those considered in \cite{bulivcek2016existence}. \emph{This is the approach we follow to obtain the a priori weighted estimates in this paper.}  The precursor to all these elliptic estimates is the seminal paper of \cite{johnlewis} which developed the method of Lipschitz truncation in the elliptic setting partially based on the ideas from \cite{acerbifusco}.

In the parabolic case, the problem was recently addressed in \cite{bulicek2018well} where the parabolic analogue of the compactness argument from \cite{bulivcek2016existence} for the linear case was developed. The weighted/unweighted estimates were first proved in \cite{MR3318165} for the linear case and a weaker version was reproved in \cite{bulicek2018well}. While the compactness arguments from \cite{bulicek2018well} were robust enough to work even in the quasilinear situation with relatively simple modifications, the a priori estimates developed in \cite{MR3318165,bulicek2018well} were specifically tailored to the linear case and hence could not be extended to handle quasilinear equations.

In this paper, we overcome this difficulty by obtaining the required a priori estimates using the ideas from \cite{adimurthi2018gradient} and combining them with several important observations regarding the weight class considered here. We then apply the modified compactness argument based on \cite{bulicek2018well} to prove the required existence result (see \cref{main1}). As discussed before, we consider the problem in the range $(p-\be_0,p)$ (see \cref{def_be_0} for the definition of $\be_0$) and along the way, impose some restrictions on the nonlinearity $\aa(\cdot,\cdot)$ and domain $\Om$. \emph{It should be noted that our restriction on the domain $\Om$ is less stringent whereas the restriction on the nonlinearity $\aa(\cdot,\cdot)$ is more than that assumed in the elliptic analogues from \cite{bulivcek2016existence}}.  One of the main tools used to obtain both the a priori estimates and the compactness argument is based on the parabolic Lipschitz truncation developed in the seminal paper of \cite{KL}.


The plan of the paper is as follows: In \cref{section2}, we collect some preliminary lemmas and assumptions and also describe the function spaces and Muckenhoupt weights, in \cref{section2.5}, we collect few more important lemmas which will be useful in proving the compactness arguments, in \cref{section3}, we describe the main theorems that will be proved, in \cref{section4}, we shall prove the main weighted a priori estimates that will be needed to prove the existence of \emph{very weak solutions}, in \cref{section6}, we shall adapt the compactness arguments from \cite{bulicek2018well} to our setting and finally in \cref{section7}, we shall apply the results from \cref{section4,section6} to prove the existence of \emph{very weak solutions} to \cref{main}.

\section{Preliminaries}
\label{section2}

The following restriction on the exponent $p$ will always be enforced:
\begin{equation}
 \label{restriction_p}
 \frac{2n}{n+2} < p < \infty.
\end{equation}

\begin{remark}The restriction in \eqref{restriction_p} is necessary when dealing with parabolic problems because,  we invariably have to deal with the $L^2$-norm of the solution which comes from the time-derivative. On the other hand, the following Sobolev  embedding $W^{1,p} \hookrightarrow L^2$ is true provided \eqref{restriction_p} holds. 
\end{remark}

\subsection{Assumptions on the Nonlinear structure}
\label{nonlinear_structure}
We shall now collect the assumptions on the nonlinear structure in \eqref{main}. 
We assume that $\aa(x,t,\nabla u)$ is a Carath\'eodory function, i.e., we have $(x,t) \mapsto \aa(x,t,\zeta)$  is measurable for every $\zeta \in \RR^n$ and 
$\zeta \mapsto \aa(x,t,\zeta)$ is continuous for almost every  $(x,t) \in \Om \times (0,T)$. We also assume $\aa(x,t,0) = 0$ and $\aa(x,t,\zeta)$ is differentiable in $\zeta$ away from the origin, i.e., $D_{\zeta}\aa(x,t,\zeta)$ exists for a.e. $(x,t) \in \RR^{n+1}$ and $\zeta \neq 0$.

We further assume that for a.e. $(x,t) \in \Om \times (0,T)$ and for any $\eta,\zeta \in \RR^n$, there exists two given positive constants $\lamot$ such that  the following bounds are satisfied   by the nonlinear structures:
\begin{equation}\label{abounded}\begin{array}{c}
\iprod{\aa(x,t,\zeta - \aa(x,t,\eta)}{\zeta - \eta}\geq  \La_0 \lbr |\zeta|^2 + |\eta|^2 \rbr^{\frac{p-2}{2}} |\zeta - \eta|^2, \\
|\aa(x,t,\zeta) - \aa(x,t,\eta)| \leq \La_1  |\zeta-\eta|\lbr |\zeta|^2 + |\eta|^2 \rbr^{\frac{p-2}{2}}. 
\end{array}
\end{equation}

Note that from the assumption $\aa(x,t,0) = 0$, we get for a.e. $(x,t) \in \RR^{n+1}$, there holds
\begin{equation*}
|\aa(x,t,\zeta)| \leq \La_1 |\zeta|^{p-1}.
\end{equation*}

\subsection{Structure of \texorpdfstring{$\Om$}.}
The domain that we consider may be non-smooth but should satisfy some regularity condition. This condition would essentially say that at each boundary point and every scale, we require the boundary of the domain to be between two hyperplanes separated by a distance  proportional to the scale.  

\begin{definition}
\label{reif_flat}
Given any $\ga \in (0,1/8]$ and $S_0 >0$, we say that $\Om$ is \gflt domain if for every $x_0 \in \pa \Om$ and every $r \in (0,S_0]$, there exists a system of coordinates $\{y_1,y_2,\ldots,y_n\}$ (possibly depending on $x_0$ and $r$) such that in this coordinate system, $x_0 =0$ and 
\[
B_r(0) \cap \{y_n > \ga r\} \subset B_r(0) \cap \Om \subset B_r(0) \cap \{y_n > -\ga r\}.
\]
\end{definition}

The class of Reifenberg flat domains is standard in obtaining Calder\'on-Zygmund type estimates, in the elliptic case, see \cite{AP2,BO1,BW-CPAM} and references therein whereas for the parabolic case, see \cite{Bui1,MR3461425,BOS1,MR2836359} and references therein. 

From the definition of \gflt domains, it is easy to see that the following property holds:
\begin{lemma}
\label{measure_density}
Suppose that $\Om$ is a \gflt domain, then there exists an $m_e = m_e(\ga,S_0,n)\in(0,1)$ such that for every $x \in {\Om}$ and every $r >0$, there holds
\begin{equation*}
|\Om^c \cap B_r(x)| \geq m_e |B_r(x)|.
\end{equation*}
\end{lemma}

\subsection{Smallness Assumption}
In order to prove the main results, we need to assume a smallness condition satisfied by $(\aa,\Om)$.
\begin{definition}\label{definition_assumption}
We say $(\aa,\Om)$ is $(\ga,S_0)$-vanishing if the following assumptions hold:
\begin{description}[leftmargin=*]
  \item[(i) Assumption on $\aa$:] For any parabolic cylinder $Q_{\rho,s}(\mfz)$ centered at $\mfz := (\mfx,\mft)\in \RR^{n+1}$, let us define the following:
  \begin{equation*}
   \Theta(\aa,Q_{\rho,s}(\mfz))(x,t) := \sup_{\zeta \in \RR^n\setminus \{0\}} \frac{\abs{\aa(x,t,\zeta) - \aaa_{B_{\rho}(\mfx)}(t,\zeta)}}{|\zeta|^{p-1}},
  \end{equation*}
where we have used the notation 
\begin{equation*}
\aaa_{B_{\rho}(\mfx)}(t,\zeta) := \fint_{B_{\rho}(\mfx)} \aa(x,t,\zeta) \ dx.
\end{equation*}
Then $\aa$ is said to be $(\ga,S_0)$-vanishing if for some $\tau \in [1,\infty)$, there holds
\begin{equation}
 \label{small_aa}
 [\aa]_{\tau,S_0} := \sup_{\substack{0<\rho\leq S_0\\0<s\leq S_0^2}} \sup_{\substack{\mfz \in  \RR^{n+1}\\\mfz:=(\mfx,\mft)}} \fiint_{Q_{\rho,s}(\mfz)} \abs{\Theta(\aa,B_{\rho}(\mfx))(z)}^{\tau} \ dz \leq \ga^{\tau}.
\end{equation}
Here we have used the notation $z := (x,t) \subset \RR^{n+1}$ and $dz := dx \ dt$.

\item[(ii) Assumption on $\pa \Om$:] We require that $\Om$ is a \gflt  in the sense of  Definition \ref{reif_flat}. 

\item[(iii) Assumptions on $\vec{h}$:]
    \label{h_hypothesis}
    In what follows, we shall assume that for any $\be \in (0,\be_0)$ (where $\be_0$ is from \cref{def_be_0}), the vector field $\vec{h}$ satisfies the following hypothesis:
    \begin{enumerate}[(a)]
        \item\label{h_hyp_1}  There exists a function $h \in L^{p-\be}(0,T; W^{1,p-\be}(\Om))$ such that $\vec{h} = \nabla h$.
        \item\label{h_hyp_2}  For the function $h$ from \cref{h_hyp_1}, it satisfies $h_t \in L^\frac{p-\beta}{p-1}(0,T;W^{-1,\frac{p-\beta}{p-1}}(\Om))$. Furthermore, there exists a vector field $\vec{H}\in L^{p-\be}(\Om_T)$ such that 
        \[
            h_t = \dv |\vec{H}|^{p-2} \vec{H} \txt{in} L^\frac{p-\beta}{p-1}(0,T;W^{-1,\frac{p-\beta}{p-1}}(\Om)).
        \]
       \item\label{h_hyp_3} For the function $h$ from \cref{h_hyp_1}, there eixsts an approximating sequence $\{h_k\}_{k\ge1}\subset C^\infty(\Om_T)$ such that 
        \[
            \nabla h_k\to \nabla h = \vec{h}\txt{ in }L^{p-\beta}(\Om_T).
         \]
    \end{enumerate}

 \end{description}
\end{definition}
  
\begin{remark}
From \eqref{abounded}, we see that $|\Theta(\aa,Q_{\rho,s}(\mfz))(x,t)| \leq 2 \La_1$, thus combining this with the assumption \eqref{small_aa}, we see from standard interpolation inequality that for any $1 \leq \mathfrak{t} <\infty$, there holds
\[
\fiint_{Q_{\rho,s}(\mfz)} |\Theta(\aa,Q_{\rho,s}(\mfz))(z)|^{\mathfrak{t}} \ dz \leq C(\ga,\La_1),
\]
with $C(\ga, \La_1) \rightarrow 0$ whenever $\ga \rightarrow 0$. 
\end{remark}

\subsection{Some results about Maximal functions}

For any $f \in L^1(\RR^{n+1})$, let us now define the strong maximal function in $\RR^{n+1}$ as follows:
\begin{equation}
 \label{par_max}
 \mm(|f|)(x,t) := \sup_{\tQ \ni(x,t)} \fiint_{\tQ} |f(y,s)| \ dy \ ds,
\end{equation}
where the supremum is  taken over all parabolic cylinders $\tQ_{a,b}$ with $a,b \in \RR^+$ such that $(x,t)\in \tQ_{a,b}$. An application of the Hardy-Littlewood maximal theorem in $x-$ and $t-$ directions shows that the Hardy-Littlewood maximal theorem still holds for this type of maximal function (see \cite[Lemma 7.9]{liebermanbook} for details):
\begin{lemma}
\label{max_bound}
 If $f \in L^1(\RR^{n+1})$, then for any $\al >0 $, there holds
 \[
  |\{ z \in \RR^{n+1} : \mm(|f|)(z) > \al\}| \leq \frac{5^{n+2}}{\al} \|f\|_{L^1(\RR^{n+1})},
 \]
 and if $f \in L^{\vartheta}(\RR^{n+1})$ for some $1 < \vartheta \leq \infty$, then there holds
 \[
  \| \mm(|f|) \|_{L^{\vartheta}(\RR^{n+1})} \leq C_{(n,\vartheta)} \| f \|_{L^{\vartheta}(\RR^{n+1})}.
 \]

\end{lemma}

\subsection{Muckenhoupt weights}

In this subsection, let us collect all the properties of the weights that will be considered in the paper. See \cite[Chapter 9]{Grafakos} for the details concerning this subsection.
\begin{definition}[Strong Muckenhoupt Weight]
\label{muck_weight}
A non negative, locally integrable function $\om$ is a \emph{strong weight} in $A_q(\RR^{n+1})$ for some $1 < q < \infty$ if 
\begin{equation*}
\sup_{\mfz \in \RR^{n+1}}\sup_{\substack{0<\rho<\infty,\\0<s<\infty}} \lbr \fiint_{Q_{\rho,s}(\mfz)} \om(z) \ dz \rbr \lbr \fiint_{Q_{\rho,s}(\mfz)}\om^{\frac{-1}{q-1}}(z) \ dz \rbr^{q-1} =: [\om]_{q} < \infty.
\end{equation*}
In the case $q =1$, we define the \emph{strong $A_1(\RR^{n+1})$ weight} to be the class of non negative, locally integrable function $\om \in A_1(\RR^{n+1})$ satisfying
\begin{equation*}
\sup_{z \in \RR^{n+1}}\sup_{\substack{0<\rho<\infty,\\0<s<\infty}} \lbr \fiint_{Q_{\rho,s}(\mfz)} \om(z) \ dz \rbr \norm{\om^{-1}}_{L^{\infty}(Q_{\rho,s}(\mfz))} =: [\om]_{1} < \infty.
\end{equation*}
The quantity $[w]_q$ for $1\leq q < \infty$ will be called as the $A_q$ constant of the weight $\om$. 
\end{definition}

We will need the following important characterization of Muckenhoupt weights:
\begin{lemma}
\label{weight_lemma}
A parabolic weight $w \in A_q$ for $1<q<\infty$ if and only if 
\[
\lbr \frac{1}{|Q|} \iint_Q f(x,t)  \ dx \ dt \rbr^q \leq \frac{c}{w(Q)} \iint_Q |f(x,t)|^q w(x,t) \ dx \ dt,
\]
holds for all non-negative, locally integrable functions $f$ and all cylinders $Q = Q_{\rho,s}$.
\end{lemma}
%

As a direct consequence of Lemma \ref{weight_lemma}, the following  Lemma holds:

\begin{lemma}\label{weight_lemma2}
Let $\om \in A_q(\RR^{n+1})$ for some $1<q<\infty$, then there exists positive constants $c = c(n,q,[\om]_q)$ and $\tau = \tau(n,q,[\om]_q) \in (0,1)$ such that 
\begin{equation*}
\frac{1}{c}\lbr \frac{|E|}{|Q|} \rbr^q \leq  \frac{\om(E)}{\om(Q)} \leq c \lbr \frac{|E|}{|Q|} \rbr^{\tau},
\end{equation*}
for all $E \subset Q$ and all parabolic cylinders $Q_{\rho,s}$. 

\end{lemma}

It is well known that the class of Muckenhoupt weights satisfy a reverse H\"older inequality (see for example \cite[Theorem 9.2.2]{Grafakos} for the details) given by
\begin{lemma}
    \label{reverse_holder_weight}
    Let $\om \in A_q$ for some $1\leq q < \infty$, then there exists constants $C=C(n,q,[w]_{A_q})$ and $\de = \de(n,q,[w]_{A_q})$ such that for every cube, there holds
    \[
        \lbr \frac{1}{|Q|} \iint_{Q} \om(z)^{1+\de} \ dz \rbr^{\frac{1}{1+\de}} \leq C \frac{1}{|Q|} \iint_{Q} \om(z) \ dz.
    \]
\end{lemma}

As a corollary, the following self-improvement property holds.
\begin{lemma}
\label{reverse_holder}
Let $1<q<\infty$ and suppose $\om \in A_q$ be a given weight, then there exists an $\ve_0 = \ve_0(n,q,[\om]_q)>0$ such that $\om \in A_{q-\ve_0}$ with the estimate $[\om]_{q-\ve_0} \leq C [\om]_q$ where $C = C{(q,n,[\om]_q)}$. 
\end{lemma}
%

We will now define the $A_{\infty}$ class as follows:
\begin{definition}
\label{a_infinity}
A weight $\om \in A_{\infty}$ if and only if there are constants $\tau_0, \tau_1 >0$ such that for every  parabolic cylinder $Q=Q_{\rho,s} \subset \RR^{n+1}$ and every measurable $E \subset Q$, there holds
\[
\om(E) \leq \tau_0 \lbr \frac{|E|}{|Q|} \rbr^{\tau_1} \om(Q).
\]
Moreover, if $\om$ is an $A_q$ weight with $[\om]_q \leq \overline{\omega}$, then the constants $\tau_0$ and $\tau_1$ can be chosen such that $\max\{ \tau_0,1/\tau_1 \} \leq c(\overline{\om},n)$. 
\end{definition}

From the general theory of Muckenhoupt weights, we see that $A_{\infty} = \bigcup_{1 \leq q <\infty} A_q$. 

%
We now have the following important bounds for the Hardy Littlewood maximal function on weighted spaces (for example, see \cite[Chapter 9]{Grafakos} for more on this).
\begin{theorem}\label{weigted_lemma3} 
Let $1<q<\infty$ and suppose that $\om\in A_q(\mathbb{R}^{n+1})$ and let $\mathcal{O}$ be a bounded open subset of $\mathbb{R}^{n+1}$. Then for any $f\in L^q_\om(\mathcal{O})$, there holds
\[
\iint_{\mathcal{O}}\lv \mm(f)\rv^q\om\ dz\le C\iint_{\mathcal{O}}\lv f\rv^q\om\ dz,
\]
where $C=C(n,q,[\om]_{q})$.
\end{theorem}


\subsection{Function Spaces}\label{function_spaces}

Let $1\leq \vt < \infty$, then $W_0^{1,\vt}(\Om)$ denotes the standard Sobolev space which is the completion of $C_c^{\infty}(\Om)$ under the $\|\cdot\|_{W^{1,\vt}}$ norm. 

The parabolic space $L^{\vt}(0,T; W^{1,\vt}(\Om))$  is the collection of measurable functions $\phi(x,t)$ such that for almost every $t \in (0,T)$, the function $x \mapsto \phi(x,t)$ belongs to $W^{1,\vt}(\Om)$ with the following norm being finite:
\[
 \| \phi\|_{L^{\vt}(0,T; W^{1,\vt}(\Om)} := \lbr \int_{0}^T \| \phi(\cdot, t) \|_{W^{1,\vt}(\Om)}^{\vt} \ dt \rbr^{\frac{1}{\vt}} < \infty.
\]

Analogously, the parabolic space $L^{\vt}(0,T; W_0^{1,\vt}(\Om))$ is the collection of measurable functions $\phi(x,t)$ such that for almost every $t \in (0,T)$, the function $x \mapsto \phi(x,t)$ belongs to $W_0^{1,\vt}(\Om)$.

Given a weight $\om \in A_{\vt}$, the weighted Lebesgue space $L^{\vt}(0,T; L^{\vt}_{\om}(\Om))$ is the set of all measurable functions $\phi: \Om_T \mapsto \RR$ satisfying 
\[
 \int_{0}^T \lbr \int_{\Om} |\phi(x,t)|^{\vt} \om(x,t)\  dx \rbr  dt < \infty.
\]

 Let us recall the following important characterization of Lebesgue spaces:
\begin{lemma}
\label{weight_level_set}
Let $\Om$ be a bounded domain in $\RR^{n}$ and let $w \in L^1(\Om_T)$ be any non-negative function, then for all $\be > \al > 1$ and any non-negative measurable function $g(x,t): \Om_T \mapsto \RR$, there holds
\[
\iint_{\Om_T} g^{\be} w(z) \ dz = \be \int_0^{\infty} \la^{\be -1} w(\{ z \in \Om_T: g(z) > \la\}) \ d\la = (\be -\al) \int_{0}^{\infty} \la^{\be-\al-1} \lbr \iint_{\{z \in \Om_T: g(z) > \la\}} g^{\al} w(z) \ dz \rbr \ d\la.
\]

\end{lemma}

\subsection{Very weak solution}

There is a well known difficulty in defining the notion of solution for \eqref{main} due to a lack of time derivative of $u$. To overcome this, one can either use Steklov average or convolution in time. In this paper, we shall use the former approach (see also \cite[Page 20, Equation (2.5)]{DiB1} for further details).

Let us first define Steklov average as follows: let $h \in (0, T)$ be any positive number, then we define
\begin{equation*}
  u_{h}(\cdot,t) := \left\{ \begin{array}{ll}
                              \hint_t^{t+h} u(\cdot, \tau) \ d\tau \quad & t\in (0,T-h), \\
                              0 & \text{else}.
                             \end{array}\right.
 \end{equation*}

\begin{definition}[Very weak solution] 
\label{very_weak_solution}
 
 Let $ \be \in (0,1)$ and $h \in (0,T)$ be given and suppose $p-\be > 1$. We  then say $u \in L^2(0,T; L^2(\Om)) \cap L^{p-\be}(0,T; W_0^{1,p-\be}(\Om))$ is a very weak solution of \eqref{main} if for any $\phi \in W_0^{1,\frac{p-\be}{1-\be}}(\Om) \cap L^{\infty}(\Om)$, the following holds:
 \begin{equation}
 \label{def_weak_solution}
  \int_{\Om \times \{t\}} \frac{d [u]_{h}}{dt} \phi + \iprod{[\aa(x,t,\nabla u)]_{h}}{\nabla \phi} \ dx = 0 \txt{for almost every}0 < t < T-h.
 \end{equation}

\end{definition}

\subsection{Notation}

\section{Some well known Lemmas}
\label{section2.5}

Let us now recall a well known compactness lemma proved in \cite{simons1987compact}.
\begin{theorem}\label{Aubin_Lions}
Let $X,B$ and $Y$ are Banach spaces such that 
\[
X\subset B\subset Y\text{ with compact embedding }X\to B.
\]
Let $1 \leq p < \infty$ be fixed,  $F$ be a bounded subset in $L^p(0,T;X)$ and $\pa_tF=\{\pa_tf\mid f\in F\}$  be bounded in $L^1(0,T;Y)$. Then $F$ is relatively compact in $L^p(0,T;B)$.
\end{theorem}

Let us recall an important version of Chacon's biting lemma proved in \cite{ball1989remarks}.
\begin{lemma}\label{l1_compact}
Let $\mathcal{O}\subset \mathbb{R}^{n+1}$ be a bounded domain and $\{f^k\}_{k\ge1}$ be a bounded sequence in $L^1(\mathcal{O})$. Then there exists a non-decreasing sequence of measurable subsets $F_j\subset\mathcal{O}$ with $\left\lv\mathcal{O}\setminus F_j\right\rv\to0$ as $j\to\infty$ such that $\{f^k\}_{k\ge1}$ is pre-compact in $L^1(F_j)$ for each $j\ge1$.
\end{lemma}


Before we conclude this subsection, let us now recall the well known Poincar\'e's inequality (see \cite[Corollary 8.2.7]{AH} for the proof):
\begin{theorem}
\label{poincare}
Let $1\leq \vt < \infty$,  and $f \in W^{1,\vt}(\tilde{\Om})$ for some bounded domain $\tilde{\Om}$ and suppose that the following measure density condition holds:
\[\abs{\{ x \in \tilde{\Om}: f(x) = 0\}} \geq m_e > 0,\] then there holds
\[
\int_{\tilde{\Om}} \abs{\frac{f}{\diam(\tilde{\Om})}}^{\vt}\ dx \leq C_{(n,\vt,m_e)} \int_{\tilde{\Om}} |\nabla f|^{\vt}  \ dx.
\]

\end{theorem}

\section{Main Theorems}
\label{section3}

Let us first state the main weighted estimate that will be obtained in this paper:
\begin{theorem}\label{weighted_estimate}
 Let $S_0 \in (0,\infty)$ be given, then there exist $\ga_0 = \ga_0(n,p,\lamot,\Om)$ and $\be_2 = \be_2(n,p,\lamot,\Om)$ such that for any $\ga \in (0,\ga_0]$, if $(\aa,\Om)$ is $(\ga,S_0)$-vanishing, then the following holds for any $\be \in (0,\be_2]$: Let $\phi \in L^1_{\loc}(\Om_T)$ be a given function such that $\mm(|\phi|\lsb{\chi}{\Om_T})(z) < \infty$ almost everywhere and denote $$w(z) :=\mm(|\phi|\lsb{\chi}{\Om_T})^{-\be}(z).$$ Furthermore, let $\vec{f}\in L^{p-\be}(\Om_T) \cap L^p_{\om}(\Om_T)$ and $ \vec{h} \in L^{p-\be}(\Om_T) \cap L^p_{\om}(\Om_T)$ be any two given vector fields such that \cref{h_hypothesis} holds and let $u \in L^{p-\be}(0,T;W_0^{1,p-\be}(\Om))$ with $|\nabla u| \in L^p_{\om}(\Om_T)$ be a distributional solution of \cref{main} in the sense of \cref{def_weak_solution}, then the following a priori estimate is satisfied
 \[
     \iint_{\Om_T} |\nabla u|^p \mm(|\phi|\lsb{\chi}{\Om_T})^{-\be}(z) \ dz \apprle_{(n,p,\lamot,\Om,\om(\Om_T))} \lbr \iint_{\Om_T} \lbr |\vec{f}|^p + |\vec{h}|^p +1\rbr\mm(|\phi|\lsb{\chi}{\Om_T})^{-\be}(z) dz \rbr^{d},
 \]
where $d = d(n,p,\be) \geq 1$ is a universal constant.
\end{theorem}

Before we state the existence result, let us first collect the restrictions on the exponent $\be_0$:
\begin{definition}[Definition of \texorpdfstring{$\be_0$}.]
\label{def_be_0}
The following restriction on the exponent $\be_0= \be_0(n,p,\lamot)$ shall be imposed. 
\begin{enumerate}
    \item\label{ve1} We need $\be_0 \leq \be_1$ so that \cref{unweighted_estimate} holds.
    \item\label{ve2} We need $\be_0 \leq \be_2$ so that \cref{weighted_estimate} holds.
    \item\label{ve3} We need $\be_0 \leq \be_4$ so that \cref{claim_weight} holds.
\end{enumerate}
\end{definition}

We are now ready to state the main theorem which is the existence of very weak solutions for \cref{main}.
\begin{theorem}
    \label{main1}
    There exists constant $\be_0= \be_0(n,p,\lamot)$ (as quantified in \cref{def_be_0}) such that for all $\be \in (0,\be_0]$ the following holds: Let $S_0>0$ be given, then there exists constants $\ga_0 = \ga_0(n,p,\lamot,\be,\Om)$ such that if $(\aa,\Om)$ is $(\ga,S_0)$-vanishing for some $\ga \in (0,\ga_0)$, then for any vector fields $\vec{f} \in L^{p-\be}(\Om_T)$ and $\vec{h} \in L^{p-\be}(\Om_T)$ satisfying \cref{h_hypothesis}, there exists a \emph{very weak solution} $u \in L^{p-\be}(0,T;W_0^{1,p-\be}(\Om))$ solving \cref{main} in the sense of \cref{def_weak_solution}.
\end{theorem}

\section{A priori estimates}
\label{section4}

In this section, we shall obtain the main a priori estimates that will be needed to prove the existence results. The first is an unweighted estimate below the natural exponent for \emph{very weak solutions} of \cref{main}. The proof of this result follows essentially as in \cite[Theorem 6.1]{adimurthi2018gradient} and the theorem reads as follows:
\begin{theorem}
    \label{unweighted_estimate}
    Let the nonlinearity $\aa$ satisfy \cref{abounded} and $\Om$ be a bounded domain that satisfies \cref{measure_density}. Then there exists $\be_1  = \be_1(n,p,\lamot,m_e) \in (0,1)$ such that for all $\be \in (0,\be_1]$ and any two given vector fields $\vec{f} \in L^{p-\be}(\Om_T)$ and $\vec{h} \in L^{p-\be}(\Om_T)$ satisfying \cref{h_hypothesis}, the following holds:  For any very weak solution $u \in L^{p-\be} (0,T; W_0^{1,p-\be}(\Om))$ solving \cref{main} in the sense of \cref{very_weak_solution}, there holds
    \[
        \iint_{\Om_T} |\nabla u|^{p-\be} \ dz \apprle_{(n,p,\lamot,m_e)} \iint_{\Om_T} \lbr |\vec{h}|^{p-\be} + |\vec{f}|^{p-\be}\rbr \ dz.
    \]
\end{theorem}

To obtain the existence of \emph{very weak solutions} to \cref{main}, we need to obtain suitable weighted esitmates which we will obtain as follows. Let us first recall several important properties of the special class of Muckenhoupt weights that we will consider. The first is a description of a subset of $A_1$ weights (see for example \cite[Theorem 9.2.7]{Grafakos} or \cite[Pages 229-230]{torchinsky1986real} for the details).
\begin{lemma}
\label{lem_weight_one}
    Let $\phi \in L^1_{\loc}(\Om_T)$ be any function such that $\mm(\phi)(z) < \infty$ almost everywhere and $\al \in (0,1)$ be a given exponent, then the following two conclusions hold: Firstly, the function $\mm(\phi)^{\al}(z) \in A_1(\RR^{n+1})$ and secondly, the $A_1$ norm of this weight is given by
    \[
        [\mm(\phi)^{\al}]_{A_1} \leq \frac{C(n)}{1-\al}.
    \]
\end{lemma}

The second lemma that we will need is a way to construct $A_p$ weights from $A_1$ weights (see for example \cite[Exercise 9.1.2]{Grafakos} for the details).
\begin{lemma}
\label{lem_weight_two}
Given any weight $ \om \in A_1$ and any $1< p < \infty$, we have $\om^{1-p} \in A_p$ with 
\[
    [\om^{1-p}]_{A_p} \leq [\om]_{A_1}^{p-1}.
\]
\end{lemma}

We shall make the following important observations that will enable us to obtain the desired weighted estimates in this section.

\begin{remark}\label{3obs_rmk}
From \cref{lem_weight_one}, \cref{lem_weight_two} and \cref{reverse_holder_weight}, we have the following three observations:
\begin{description}
    \item[Observation 1:] Given any function $\phi \in L^1_{\loc}(\RR^{n+1})$, the function given by $\mm(\phi)^{\al(1-p)}(z)$ for some $\al \in (0,1/2]$ and $1 < p < \infty$ is an $A_p$ weight.
    \item[Observation 2:] The $A_p$ constant of the weight given by $\mm(\phi)^{\al(1-p)}$ is independent of the function $\phi$ and $\al \in [0,1/2]$ and depends only $n$ and $p$. In particular,
                   \begin{equation}\label{obs_we1}
                       [\mm(\phi)^{\al(1-p)}]_{A_p} \overset{\text{\cref{lem_weight_two}}}{\leq} [\mm(\phi)^{\al}]_{A_1}^{p-1} \overset{\text{\cref{lem_weight_one}}}{\leq} \lbr \frac{C(n)}{1/2}\rbr^{p-1} = C(n,p).
                   \end{equation}
   \item[Observation 3: ] As a consequence of  \cref{obs_we1} and \cref{reverse_holder_weight}, we see that the weight given by $\mm(\phi)^{\al(1-p)}$ is in $A_{p-\de_0}$ for some universal constant $\de_0 = \de_0(n,p)$. In particular, the self improvement property is independent of $\phi$ and $\al$.
\end{description}
\end{remark}

\subsection{Covering arguments needed for the proof of \texorpdfstring{\cref{weighted_estimate}}.}

The proof of this theorem crucially uses some of the a priori estimates proved in \cite{adimurthi2018gradient,byun2017weighted}.  The a priori estimates below the natural exponent are proved in \cite{adimurthi2018gradient} and the covering argument we follow is obtained in \cite{byun2017weighted} based on the techniques developed in the \cite{MR2286632}.

    Let $\ve \in (0,1)$ be given (which will eventually be chosen) and for $\be \in (0,\be_2)$ (note that $\be_2$ will eventually be fixed to depend on data), consider the weight given by 
    \begin{equation}\label{def_weight_est}
        \mm(|\phi|\lsb{\chi}{\Om_T})^{-\be} (z)  \txt{for some} \phi \in L^1_{\loc}(\Om_T).
    \end{equation}

    For a given $\ve \in (0,1)$, we can find a $\de_{\ve} = \de_{\ve}(n,p,\lamot,\ve)\in(0,1)$ such that \cite[Theorem 5.6 and Theorem 5.7]{adimurthi2018gradient} holds.
    Here we take the largest possible choice of $\de_{\ve}$. 
    
    \begin{claim}\label{claim1}
        There exists $\tilde{\be}_1 := \frac{\de_{\ve}}{2(p-\de_{\ve})}$ such that for all $\be \in (0,\tilde{\be}_1)$, we have
        \[
          \left[\mm(|\phi|\lsb{\chi}{\Om_T})^{-\be}\right]_{A_{\frac{p}{p-\de_{\ve}}}} \leq C(n).
          \]
    \end{claim}
    \begin{proof}[Proof of \cref{claim1}]
        We have the following sequence of estimates:
        \[
            \begin{array}{rcl}
                \left[\mm(|\phi|\lsb{\chi}{\Om_T})^{-\be}\right]_{A_{\frac{p}{p-\de_{\ve}}}} & = & \left[\mm(|\phi|\lsb{\chi}{\Om_T})^{\frac{\be(p-\de_{\ve})}{\de_{\ve}}\lbr 1-\frac{p}{p-\de_{\ve}}\rbr}\right]_{A_{\frac{p}{p-\de_{\ve}}}}\\
                & \overset{\text{\cref{lem_weight_two}}}{\leq} & \left[\mm(|\phi|\lsb{\chi}{\Om_T})^{\frac{\be(p-\de_{\ve})}{\de_{\ve}}}\right]_{A_{1}}^{\lbr \frac{p}{p-\de_{\ve}}-1\rbr}\\
                & \overset{\text{\cref{lem_weight_one}}}{\leq} & \lbr \frac{C(n)}{1-\frac{\be(p-\de_{\ve})}{\de_{\ve}}}\rbr^{\lbr \frac{p}{p-\de_{\ve}}-1\rbr}\\
                & \overset{\redlabel{cl1}{a}}{\leq} & C(n).
            \end{array}
        \]
    To obtain \redref{cl1}{a}, from the choice of $\tilde{\be}_1$, we see that  $\frac{\be(p-\de_{\ve})}{\de_{\ve}} \leq \frac12$ and $\frac{p}{p-\de_{\ve}}-1 = \frac{\de_{\ve}}{p-\de_{\ve}} \leq 1$.
    \end{proof}
 
     \begin{claim}
     \label{claim2}
         For any $\be \in (0,\tilde{\be}_2)$, with $\tilde{\be}_2 := \frac{p-1}{2}$, we have
         \[
             \left[\mm(|\phi|\lsb{\chi}{\Om_T})^{-\be}\right]_{A_p} \leq C(n,p).
         \]
     \end{claim}
     \begin{proof}[Proof of \cref{claim2}]
         We have the following sequence of estimates:
        \[
            \begin{array}{rcl}
                \left[\mm(|\phi|\lsb{\chi}{\Om_T})^{-\be}\right]_{A_{p}} & = & \left[\mm(|\phi|\lsb{\chi}{\Om_T})^{\frac{\be}{p-1}\lbr 1-p\rbr}\right]_{A_{p}}\\
                & \overset{\text{\cref{lem_weight_two}}}{\leq} & \left[\mm(|\phi|\lsb{\chi}{\Om_T})^{\frac{\be}{p-1}}\right]_{A_{1}}^{p-1}\\
                & \overset{\text{\cref{lem_weight_one}}}{\leq} & \lbr \frac{C(n)}{1-\frac{\be}{p-1}}\rbr^{p-1}\\
                & \overset{\redlabel{cl2}{a}}{\leq} & C(n,p).
            \end{array}
        \]
        To obtain \redref{cl2}{a}, we made use of the fact that $2\be \leq p-1$ from the hypothesis.
     \end{proof}

     From \cref{claim1} and \cref{reverse_holder}, we see that there exists constants $\tilde{\de} = \tilde{\de}(n,p,\de_{\ve})$ and $C= C(n,p,\de_{\ve})$ such that 
     \begin{equation}\label{self_impr_bound}
         \left[\mm(|\phi|\lsb{\chi}{\Om_T})^{-\be}\right]_{A_{\frac{p}{p-\de_{\ve}}-\tilde{\de}}} \leq C \left[\mm(|\phi|\lsb{\chi}{\Om_T})^{-\be}\right]_{A_{\frac{p}{p-\de_{\ve}}}} \leq C(n,p,\de_{\ve}).
     \end{equation}
     It is important to note that 
     \begin{equation}\label{5.4}
         1< \frac{p}{p-\de_{\ve}} - \tilde{\de} < \frac{p}{p-\de_{\ve}}.
     \end{equation}

     \begin{claim}
         \label{claim3}
         If we further restrict $\be \in (0,\tilde{\be}_3)$ where $\tilde{\be}_3 := \frac{\de_{\ve}-\tilde{\de}(p-\de_{\ve})}{2(p-\de_{\ve})}$, then we have the following ameliorated bound
         \[
             \left[\mm(|\phi|\lsb{\chi}{\Om_T})^{-\be}\right]_{A_{\frac{p}{p-\de_{\ve}}-\tilde{\de}}} \leq C(n).
         \]
     \end{claim}
     \begin{proof}[Proof of \cref{claim3}]
         We have the following sequence of estimates:
         \[
            \begin{array}{rcl}
                \left[\mm(|\phi|\lsb{\chi}{\Om_T})^{-\be}\right]_{A_{\frac{p}{p-\de_{\ve}}-\tilde{\de}}} & = & \left[\mm(|\phi|\lsb{\chi}{\Om_T})^{\frac{\be}{\frac{p}{p-\de_{\ve}}-\tilde{\de}-1}\lbr 1-\frac{p}{p-\de_{\ve}}+\tilde{\de}\rbr}\right]_{A_{\frac{p}{p-\de_{\ve}}}}\\
                & \overset{\text{\cref{lem_weight_two}}}{\leq} & \left[\mm(|\phi|\lsb{\chi}{\Om_T})^{\frac{\be}{\frac{p}{p-\de_{\ve}}-\tilde{\de}-1}}\right]_{A_{1}}^{\lbr \frac{p}{p-\de_{\ve}}-\tilde{\de}-1\rbr}\\
                & \overset{\text{\cref{lem_weight_one}}}{\leq} & \lbr \frac{C(n)}{1-\frac{\be}{\frac{p}{p-\de_{\ve}}-\tilde{\de}-1}}\rbr^{\lbr \frac{\de_{\ve}}{p-\de_{\ve}}-\tilde{\de}\rbr}\\
                & \overset{\redlabel{cl3}{a}}{\leq} & C(n).
            \end{array}
        \]
        To obtain \redref{cl3}{a}, we used the hypothesis which implies $\frac{\be}{\frac{p}{p-\de_{\ve}}-\tilde{\de}-1} \leq \frac12$ along with the  bound $\frac{\de_{\ve}}{p-\de_{\ve}}-\tilde{\de} \leq 2$.
     \end{proof}

     Let us now take $\be \leq \min\{\tilde{\be}_1,\tilde{\be}_2,\tilde{\be}_3\}$ where $\tilde{\be}_1$ is such that \cref{claim1} holds, $\tilde{\be}_2$ is such that \cref{claim2} holds and $\tilde{\be}_3$ is such that \cref{claim3} holds. We now define
\begin{equation}\label{al_0_weight}
\alpha^\frac{p-\de_{\ve}}{d}:=\fiint_{\Om_T}\lv \na u\rv^{p-\de_{\ve}}+\left(\frac{\lv\vec{h}\rv+\lv\vec{f}\rv}{\ga}\right)^{p-\de_{\ve}}\ dz,
\end{equation}
where $d$ is defined to be
\[
d:=
\begin{cases}
\frac{p-\de_{\ve}}{2-\de_{\ve}}&\text{ if }p\ge2,\\
\frac{2(p-\de_{\ve})}{2p-2\de_{\ve}+np-2n}&\text{ if }\frac{2n}{n+2}<p<2,
\end{cases}
\]
and for any $\la>0$, denote
\begin{equation}\label{def_om}
E_\la:=\{z\in\Om_T\mid \lv \na u(z)\rv>\la\}\txt{ and }\om(z):=\mm(|\phi|\lsb{\chi}{\Om_T})^{-\beta}(z).
\end{equation}

We have the following important lemma proved in \cite[Lemma 7.2]{adimurthi2018gradient}:
\begin{lemma}
    \label{lemma5.7}
    Let $\gamma\in(0,1)$ be any constant and $S_0$ be from \cref{definition_assumption}, then for any $\la\ge c_e\alpha$ where 
\[
c_e:=\left[\left(\frac{16}{7}\right)^n\frac{\lv \Om_T\rv}{\lv B_1\rv S_0^{n+2}}\right], 
\]
there exists a family of disjoint cylinders $\{K_{r_i}^\la(z_i)\}_{i\in\mathbb{N}}$ with $z_i\in E_\la$ and $r_i\in (0,S_0)$ such that
\begin{equation}\label{bounds_many}\begin{array}{c}
\fiint_{K_{r_i}^{\la}(z_i)}  |\nabla u|^{p-\de_{\ve}} + \left(\frac{\lv\vec{h}\rv+\lv\vec{f}\rv}{\ga}\right)^{p-\de_{\ve}}\ dz = \la^{p-\de_{\ve}}, \\
\fiint_{K_{r}^{\la}(z_i)}  |\nabla u|^{p-\de_{\ve}}+ \left(\frac{\lv\vec{h}\rv+\lv\vec{f}\rv}{\ga}\right)^{p-\de_{\ve}}  \ dz < \la^{p-\de_{\ve}} \qquad \text{for every} \ r> r_i, \\
\elam \subset \bigcup_{i\in\NN} K_{5r_i}^{\la}(z_i).
\end{array}\end{equation}

\end{lemma}

Using the previous lemma, we have the following important weighted estimate:
\begin{lemma}\label{5.7}
There exists a constant $c^*=c^*(n,p)$ such that 
\[
 \om(K_{r_i}^{\la}(z_i)) \leq \frac{C(n,p)}{\la^{p-(p-\de_{\ve})\tilde{\de}}}   \lbr[[] \iint\limits_{K_{r_i}^{\la}(z_i) \cap \{|\nabla u| > \frac{\la}{4c^{\ast}}\}} |\nabla u|^{p-(p-\de_{\ve})\tilde{\de}} \om \ dz + \iint\limits_{K_{r_i}^{\la}(z_i) \cap \{|\vec{h} |+|\vec{f} | > \frac{\ga \la}{4c^{\ast}}\}} \left(\frac{\lv\vec{h}\rv+\lv\vec{f}\rv}{\ga}\right) ^{p-(p-\de_{\ve})\tilde{\de}} \om \ dz\rbr[]].
\]
\end{lemma}
\begin{proof}
    Applying H\"older's inequality, we have
    \begin{equation}\label{eq5.7}
\begin{array}{rcl}
\fiint_{K_{r_i}^{\la}(z_i)}|\nabla u|^{p-\de_{\ve}} \ dz & = & \fiint_{K_{r_i}^{\la}(z_i)}|\nabla u|^{p-\de_{\ve}}\mm(\lv \phi\rv\rchi_{\Om_T})^{-\beta\frac{p-\de_{\ve}}{p-(p-\de_{\ve})\tilde{\de}}}\mm(\lv \phi\rv\rchi_{\Om_T})^{\beta\frac{p-\de_{\ve}}{p-(p-\de_{\ve})\tilde{\de}}} \ dz\vsp\\
&\le &  \left(\fiint_{K_{r_i}^{\la}(z_i)}|\nabla u|^{p-(p-\de_{\ve})\tilde{\de}}\mm(\lv \phi\rv\rchi_{\Om_T})^{-\beta}\ dz\right)^\frac{p-\de_{\ve}}{p-(p-\de_{\ve})\tilde{\de}} \\
&& \qquad \times\left(\fiint_{K_{r_i}^{\la}(z_i)}\mm(\lv \phi\rv\rchi_{\Om_T})^\frac{\beta(p-\de_{\ve})}{\de_{\ve}-(p-\de_{\ve})\tilde{\de}}\ dz\right)^\frac{\de_{\ve}-(p-\de_{\ve})\tilde{\de}}{p-(p-\de_{\ve})\tilde{\de}}.
\end{array}
\end{equation}

From \cref{claim3} and \cref{muck_weight}, we see that 
\begin{equation}
    \label{eq5.8}
    \begin{array}{rcl}
\left(\fiint_{K_{r_i}^{\la}(z_i)}\mm(\lv \phi\rv\rchi_{\Om_T})^\frac{\beta(p-\delta_{\ve})}{\delta_{\ve}-(p-\delta_{\ve})\tilde{\de}}\ dz\right)^\frac{\delta_{\ve}-(p-\delta_{\ve})\tilde{\de}}{p-\delta_{\ve}}&\le &  \left[\mm(|\phi|\lsb{\chi}{\Om_T})^{-\beta}\right]_{A_{\frac{p}{p-\delta_{\ve}}-\tilde{\de}}}\left(\fiint_{K_{r_i}^{\la}(z_i)}\mm(\lv \phi\rv\rchi_{\Om_T})^{-\beta}\ dz\right)^{-1}\vsp\\
&\le & C(n,p) \frac{|K_{r_i}^\la(z_i)|}{\om(K_{r_i}^\la(z_i))}.
\end{array}
\end{equation}
Analogously, we can estimate $\fiint_{K_{r_i}^{\la}(z_i)}\left(\frac{\lv\vec{h}\rv+\lv\vec{f}\rv}{\ga}\right)^{p-\de_{\ve}} \ dz$. Thus combining \cref{eq5.7} and \cref{eq5.8} with \cref{bounds_many}, we get
\begin{equation}\label{eq5.9}
\begin{array}{ll}
&\la^{p-(p-\delta_{\ve})\tilde{\de}}\om(K_{r_i}^\la(z_i))\le C(n,p)\iint_{K_{r_i}^{\la}(z_i)}\lbr |\nabla u|^{p-(p-\delta_{\ve})\tilde{\de}}+\left(\frac{\lv\vec{h}\rv+\lv\vec{f}\rv}{\ga}\right)^{p-(p-\delta_{\ve})\tilde{\de}}\rbr\om\ dz.
\end{array}
\end{equation}
We take  
\begin{equation}\label{def_c_star}
    \tilde{c}_* := C(n,p)^\frac{1}{p-(p-\de_{\ve})\tilde{\de}} \leq C(n,p)^p =:c^*,
\end{equation}
then from a simple calculation, we see that
\begin{equation}\label{eq5.10}
\begin{array}{rcl}
\iint\limits_{K_{r_i}^{\la}(z_i)} |\nabla u|^{p-(p-\delta_{\ve})\tilde{\de}} \om\ dz
&  \le &    \left(\frac{\la}{4\tilde{c}_*}\right)^{p-(p-\delta_{\ve})\tilde{\de}}\om(K_{r_i}^\la(z_i))+\iint\limits_{K_{r_i}^{\la}(z_i)\cap\{\lv \na u\rv>\frac{\la}{4\tilde{c}_*}\}}|\nabla u|^{p-(p-\delta_{\ve})\tilde{\de}} \om\ dz \\
&  \overset{\cref{def_c_star}}{\le} &    \left(\frac{\la}{4\tilde{c}_*}\right)^{p-(p-\delta_{\ve})\tilde{\de}}\om(K_{r_i}^\la(z_i))+\iint\limits_{K_{r_i}^{\la}(z_i)\cap\{\lv \na u\rv>\frac{\la}{4c^*}\}}|\nabla u|^{p-(p-\delta_{\ve})\tilde{\de}} \om\ dz.
 \end{array}
\end{equation}
An analogous estimate also holds for the second term on the right hand side of \cref{eq5.9}. Thus combining \cref{eq5.9} with \cref{eq5.10}, the proof of the lemma follows.
\end{proof}

We now have the following important lemma proved in \cite[Lemma 4.3]{byun2013weighted}.
\begin{lemma}
\label{lemma7.4}
For any $\la \geq c_e \al$ as in \cref{lemma5.7}, there exists a constant $N = N(\lamot,n,p)>1$ such that for any $\ve \in (0,1)$, there exists a small $\ga = \ga(\lamot,\ve,n,p)$ such that if $(\aa,\Om)$ is $(\ga,S_0)$-vanishing in the sense of \cref{definition_assumption}, then there holds
\[
\frac{\abs{\{z \in K_{5r_i}^{\la}(z_i): |\nabla u(z)| > 2N\la\}}}{|K_{r_i}^{\la}(z_i)|} \leq c_{(\lamot,n,p)} \ve^{p-1}.
\]
\end{lemma}

We now have the following important weighted estimates on the level sets:
\begin{lemma}\label{lemma5.11}
Let the notation from  \cref{def_om} be in force and $c^*$ be as obtained in \cref{5.7}. Furthermore let $N$ be given from \cref{lemma7.4}, then for any $\la \geq c_e \al$ as in \cref{lemma5.7} and  for any $\ve \in (0,1)$, there holds
\[
\om(E_{2N\la})\apprle_{(n,p,\lamot)}\frac{\ve^{(p-1)\tau_1}}{\la^{p-(p-\delta_{\ve})\tilde{\de}}}\left[\iint\limits_{\Om_T\cap\{\lv \na u\rv>\frac{\la}{4c^*}\}}\lv \na u \rv^{p-(p-\delta_{\ve})\tilde{\de}}\om\ dz+\iint\limits_{\Om_T \cap \{|\vec{h} |+|\vec{f} | > \frac{\ga \la}{4c^{\ast}}\}} \left(\frac{\lv\vec{h}\rv+\lv\vec{f}\rv}{\ga}\right) ^{p-(p-\delta_{\ve})\tilde{\de}} \om \ dz\right].
\]
Here $\tau_1=\tau_1(n,p)$ is the exponent from \cref{a_infinity} applied with \cref{claim2} under consideration.
\end{lemma}
\begin{proof}
    We observe that 
    \[
\begin{array}{rcl}
\om(E_{2N\la})
&\le &  \sum_{i\ge1}\om\lbr\{z\in K_{5r_i}^\la(z_i) : \lv \na u \rv>2N\la\}\rbr\\
&\overset{\text{\cref{weight_lemma2}}}{\leq}& C{(n,p)}\sum_{i\ge1}\left(\frac{\lv K_{5r_i}^\la\cap E_{2N\la}\rv}{\lv K_{r_i}^\la\rv}\right)^{\tau_1}\om(K_{r_i}^\la(z_i))\vsp\\
&\overset{\text{\cref{lemma7.4}}}{\leq}& C{(n,p,\lamot)}\ep^{(p-1)\tau_1}\sum_{i\ge 1}\om(K_{r_i}^\la(z_i)).
\end{array}
\]
Combining the above estimate with \cref{lemma5.7} gives the proof of the lemma.
\end{proof}

\subsection{Proof of \texorpdfstring{\cref{weighted_estimate}}.}
We are now ready to combine all the estimates to prove the main theorem:

Let $c_e$ be as given in \cref{lemma5.7} and $\al$ be from \cref{al_0_weight}, then from \cref{weight_level_set}, we get
\[
\begin{array}{rcl}
\iint_{\Om_T}\lv \na u \rv^p\om\ dz &= & p\int_0^{c_e\alpha_0}(2N\la)^{p-1}\om(\{z\in\Om_T\mid \lv \na u\rv>2N\la\})d(2N\la)\vsp\\
&& +p\int_{c_e\alpha_0}^\infty(2N\la)^{p-1}\om(\{z\in\Om_T\mid \lv \na u\rv>2N\la\})d(2N\la)\vsp\\
&=:& I+II.
\end{array}
\]
\begin{description}
    \item[Estimate for $I$:] This term is estimated as follows:
    \begin{equation*}
\begin{array}{rcl}
I& \apprle &  \om(\Om_T)\alpha^p \overset{\cref{bounds_many}}{=} \om(\Om_T) \lbr \fiint_{\Om_T} \lbr[[] |\nabla u|^{p-\de_{\ve}}  + \left(\frac{\lv\vec{h}\rv+\lv\vec{f}\rv}{\ga}\right)^{p-\de_{\ve}} \rbr[]] \ dz \rbr^{\frac{pd}{p-\de_{\ve}}}\\
&\overset{\text{\cref{unweighted_estimate}}}{\apprle} & \om(\Om_T)\left[\fiint_{\Om_T}\left(\lv \vec{h}\rv+\lv \vec{f}\rv\right)^{p-\delta_{\ve}}\ dz\right]^\frac{pd}{p-\de_{\ve}}\\
&\overset{\text{\cref{weight_lemma}}}{\apprle}& C{(n,p,\La_0,\La_1)}\om(\Om_T)^{1-d}\left[\iint_{\Om_T}\left(\lv \vec{h}\rv+\lv \vec{f}\rv\right)^{p}\om\ dz\right]^{d}\\
&\leq & C{(n,p,\La_0,\La_1,\om(\Om_T))}  \left[\iint_{\Om_T}\left(\lv \vec{h}\rv^p+\lv \vec{f}\rv^p+1\right)\om\ dz\right]^{d}.
\end{array}
\end{equation*}
Note that the constant in the above estimate is independent of $\de_{\ve}$ since we have $p-1 \leq p-\de_{\ve} \leq p$.

    \item[Estimate for $II$:] We estimate this term as follows:
    \begin{equation}\label{eq5.13}
\begin{array}{rcl}
II &\overset{\text{\cref{lemma5.11}}}{\apprle} &  \ep^{(p-1)\tau_1}\int_{c_e\alpha}^\infty\frac{\la^{p-1}}{\la^{p-(p-\delta_{\ve})\tilde{\de}}}\iint\limits_{\Om_T\cap\{\lv \na u\rv>\frac{\la}{4c^*}\}}\lv \na u \rv^{p-(p-\delta_{\ve})\tilde{\de}}\om \ dz\  d\la \\
&& +  \ep^{(p-1)\tau_1}\int_{c_e\alpha_0}^\infty\frac{\la^{p-1}}{\la^{p-(p-\delta_{\ve})\tilde{\de}}}\iint\limits_{\Om_T \cap \{|\vec{h} |+|\vec{f} | > \frac{\ga \la}{4c^{\ast}}\}} \left(\frac{\lv\vec{h}\rv+\lv\vec{f}\rv}{\ga}\right) ^{p-(p-\delta_{\ve})\tilde{\de}} \om \ dz\  d\la\\
&\leq & C{(n,p,\lamot)}  \left[\ep^{(p-1)\tau_1}\iint_{\Om_T}\lv \na u\rv^p\om\ dz+ C(\ve,n,p)\iint_{\Om_T}\left(\lv\vec{h}\rv+\lv \vec{f}\rv \right)^p\om\ dz\right].
\end{array}
\end{equation}
We now choose $\ve = \ve(n,p,\lamot)$ small such that $C(n,p,\lamot) \ve^{(p-1)\tau_1} = \frac12$ where $C(n,p,\lamot)$ is the constant appearing in the above inequality.
\end{description}

Once we choose $\ve = \ve(n,p,\lamot)$ based on \cref{eq5.13}, this fixes the choice of $\de_{\ve} = \de_{\ve}(n,p,\lamot)$ which in turn fixes $\tilde{\de}$. Now we take 
\[
    \be_2 := \min\{\tilde{\be}_1,\tilde{\be}_2,\tilde{\be}_3,\de_{\ve}, \be_1\},
\]
where $\tilde{\be}_1$ is from \cref{claim1}, $\tilde{\be}_2$ is from \cref{claim2}, $\tilde{\be}_3$ is from \cref{claim3} and $\be_1$ is from \cref{unweighted_estimate}.
This completes the proof of the theorem.

\section{Modified compactness theory from \texorpdfstring{\cite{bulicek2018well}}. }
\label{section6}

Let us first recall a  Whitney type  decomposition Lemma proved in  \cite[Lemma 3.1]{diening2010existence} or \cite[Chapter 3]{bogelein2013regularity}:
\begin{lemma}
\label{whitney_decomposition}

Let $\mathbb{E}$ be any closed set and  $\la \in (0,\infty)$ be a fixed constant. Define $\kappa := \la^{2-p}$, then there exists an $\ka$-parabolic Whitney covering  $\{Q_i(z_i)\}$ of $\mathbb{E}^c$ in the following sense:
\begin{description}
  \descitem{W1} $Q_j(z_j) = B_j(x_j) \times I_j(t_j)$ where $B_j(x_j) = B_{r_j}(x_j)$ and $I_j(t_j) = (t_j - \ka r_j^2, t_j + \ka r_j^2)$. 
  \descitem{W2}  $d_{\la}(z_j,\mathbb{E}) = 16r_j$.
  \descitem{W3} $\bigcup_j \frac12 Q_j(z_j) = \mathbb{E}^c$.
  \descitem{W4} for all $j \in \NN$, we have $8Q_j \subset \mathbb{E}^c$ and $16Q_j \cap \mathbb{E} \neq \emptyset$.
  \descitem{W5} if $Q_j \cap Q_k \neq \emptyset$, then $\frac12 r_k \leq r_j \leq 2r_k$.
  \descitem{W6} $\frac14 Q_j \cap \frac14Q_k = \emptyset$ for all $j \neq k$.
  \descitem{W7} $\sum_j \lsb{\chi}{4Q_j}(z) \leq c(n)$ for all $z \in \mathbb{E}^c$.
    \end{description}

    For a fixed $k \in \NN$, let us define $A_k := \left\{ j \in \NN: \frac34Q_k \cap \frac34Q_j \neq \emptyset\right\}$,  then we have
    \begin{description}
  \descitem{W8} For any $i \in \NN$, we have $\# A_i \leq c(n)$.
  \descitem{W9} Let $i \in \NN$ be given and let  $j \in A_i$, then $\max \{ |Q_j|, |Q_i|\} \leq C(n) |Q_j \cap Q_i|.$
  \descitem{W10}  Let $i \in \NN$ be given and let  $j \in A_i$, then $ \max \{ |Q_j|, |Q_i|\} \leq \left|\frac34Q_j \cap \frac34Q_i\right|.$
  \descitem{W11} Let $i \in \NN$ be given, then for any $j \in A_i$, we have $\frac34Q_j \subset 4Q_i$.
  \end{description}
  \end{lemma}
  
  Subordinate to the above Whitney covering, we have an associated partition of unity which we recall in the following lemma.
  \begin{lemma}
  \label{partition_unity}
  Associated to the covering given in \cref{whitney_decomposition}, there exists functions $\{ \Psi_j\}_{j \in \NN} \in C_c^{\infty}\lbr\frac34Q_j\rbr$ such that the following holds:
  \begin{description}
  \descitem{W12} $\lsb{\chi}{\frac12Q_j} \leq \Psi_j \leq \lsb{\chi}{\frac34Q_j}$.
  \descitem{W13} $\|\Psi_j\|_{\infty} + r_j \| \nabla \Psi_j\|_{\infty} + r_j^2 \| \nabla^2 \Psi_j\|_{\infty} + \la^{2-p} r_j^2 \| \pa_t \Psi_j\|_{\infty} \leq C(n)$.
  \descitem{W14} Let $i \in \NN$ be given, then $\sum_{j \in A_i} \Psi_j(z) = 1$  for all $z \in \frac34Q_i$.
\end{description}
\end{lemma}

In this section, let us take any exponent $q$ such that 
\begin{equation}\label{def_q}
\max\{1,p-1\} \leq q < p,
\end{equation}
and let us denote 
\[
    \rho := \diam (\Om).
\]

We consider following problem: Let $\vec{F},\vec{H}\in L^{p}(\Om_T,\mathbb{R}^n)$ be given and suppose that $w\in C(0,T;L^{2}(\Om))\cap L^p(0,T;W_0^{1,p}(\Om))$ be the unique weak solution of
\begin{equation}\label{given}
\begin{cases}
w_t-\dv\aa(z,\nabla w+\vec{H})=\dv \lvert \vec{F}\rvert^{p-2}\vec{F}&\text{ in }\Om_T,\\
w=0&\text{ on }\pa_p\Om_T.\\
\end{cases}
\end{equation}

Let us now define the following function:
\begin{equation}\label{maxim}
g(z):=\mathcal{M}\lt(\left[\lv \na w\rv^{q}+\lv \vec{F}\rv^{q}+\lv \vec{H}\rv^{q}\right]\rchi_{\Om_T}\rt)^\frac{1}{q}(z),
\end{equation}
where $\mathcal{M}$ is the Hardy Littlewood maximal function defined in \eqref{par_max}.

For a fixed $\la>0$, let us define the good set by
\begin{equation}\label{elam}
E_\la:=\{z\in\mathbb{R}^{n+1}\mid g(z)\le \la\}.
\end{equation}
Since $E_\la^c$ is open, from \cref{whitney_decomposition} and \cref{partition_unity}, we define following extension:
\begin{equation}\label{lipschitz_extension}
\varphi_{\la,h}(z):=w_h(z)-\sum_i\Psi(z)(w_h(z)-\phi_h^i),
\end{equation}
where
\begin{equation*}
\phi_h^i:=
\begin{cases}
\frac{1}{\lV \Psi_i\rV_{L^1(\frac{3}{4}Q_i)}}\iint_{\frac{3}{4}Q_i}w_h(z)\Psi_i(z)\rchi_{[0,T]}\ dz&\text{ if }\frac{3}{4}Q_i\subset \Om\times (0,\infty),\\
0&\text{ else.}
\end{cases}
\end{equation*}

\begin{lemma}
We have the following estimates for the function constructed in \cref{lipschitz_extension}:
\begin{enumerate}[(i)]
    \item\label{item1} For any $z \in E_{\la}^c$, there holds
    \[
\lv \varphi_{\la,h}(z)\rv\apprle_{(n,p,q,\lamot)}\rho\la.
\]
\item\label{item2} For a given $i \in \NN$ and any $j \in A_i$, there holds
\begin{equation}\label{diff}
\lv \phi_h^i-\phi_h^j\rv\apprle_{(n,p,q,\lamot)}\min\{\rho,r_i\}\la.
\end{equation}
\item\label{item3} Given any  $z \in \elam^c$, we have $z \in \frac34Q_i$ for some $i \in \NN$. Then there holds
 \begin{equation}
  \label{3.34}
  |\nabla \varphi_{\la,h}(z)| \leq_{(n,p,q,\lamot)} \la.
 \end{equation}

\item\label{item4} For any $\vartheta \geq 1$,  we have the following bound:
 \begin{equation*}
  \iint_{\Om_T\setminus\elam } |\varphi_{\la,h}(z)|^{\vartheta} \ dz \apprle_{(n,p,q,\lamot)} \iint_{\Om_T\setminus\elam } |w_h(z)|^{\vartheta}\ dz.
 \end{equation*}
 
\end{enumerate}
\end{lemma}
\begin{proof}
    The proof of \cref{item1} follows from \cite[Lemma 4.9]{adimurthi2018gradient}, the proof of \cref{item2} follows from \cite[Lemma 4.10]{adimurthi2018gradient},  the proof of \cref{item3} follows from \cite[Lemma 4.11]{adimurthi2018gradient} and finally the proof of \cref{item4} follows from \cite[Lemma 4.14]{adimurthi2018gradient}.
\end{proof}

We now prove an important pointwise estimate which follows the same idea as in \cite[Lemma 3.1.1]{bulicek2018well}:

\begin{lemma}\label{pointwise_poincare}
Let $i\in\NN$ and suppose $\frac{3}{4}Q_i\subset\Om\times(0,\infty)$, then for any $j\in A_i$ and a.e. $(x,t)=z\in\frac{3}{4}Q_i$, we have
\[
\frac{\lvert w_h(z)-\phi_h^j\rv}{r_i}\apprle_{(n,p,q,\lamot)}\la .
\]
\end{lemma}
\begin{proof}
Without loss of generality we only need to prove
\[
\frac{\lvert w_h(z)-\phi_h^i\rv}{r_i}\apprle_{(n,p,q,\lamot)}\la,
\]
since otherwise, we can apply triangle inequality and \cref{diff} to get
\[
\frac{\lvert w_h(z)-\phi_h^j\rv}{r_i}\le\frac{\lvert w(z)-\phi_h^i\rv}{r_i}+\frac{\lvert \phi_h^i-\phi_h^j\rv}{r_i}\apprle \frac{\lvert w_h(z)-\phi_h^i\rv}{r_i}+\la.
\]

We shall now split the proof into three cases, depending on where the cylinder $\frac34Q_i$ lies. 
\begin{description}[leftmargin=*]
    \item[Case $\frac34Q_i \subset \Om \times (0,T)$:]  In this case, let us fix any $z_0 = (x_0,t_0) \in \frac34Q_i$, from which we get
    \begin{equation}\label{6.77}
\begin{array}{rcl}
 w_h(z_0)-\phi_h^i
&=&\frac{1}{\lVert \Psi_i\rV_{L^1(\frac{3}{4}Q_i)}} \int_0^1 \frac{d}{ds}\left(\iint_{\frac{3}{4}Q_i} w_h\lbr x_0+s(\tx-x_0),t_0+s^2(\tlt-t_0)\rbr\Psi_i(\tz) \ d\tz\right)\ ds\\
&=&\frac{1}{\lVert \Psi_i\rV_{L^1(\frac{3}{4}Q_i)}}\int_0^1\iint_{\frac{3}{4}Q_i}\iprod{\nabla  w_h\lbr x_0+s(\tx-x_0),t_0+s^2(\tlt-t_0)\rbr}{\tx-x_0}\Psi_i(\tz)\ d\tz\ ds\\
&&+\frac{1}{\lVert \Psi_i\rV_{L^1(\frac{3}{4}Q_i)}}\int_0^1\iint_{\frac{3}{4}Q_i}\pa_t w_h\lbr x_0+s(\tx-x_0),t_0+s^2(\tlt-t_0)\rbr\ 2s(\tlt-t_0)\Psi_i(\tz)\ d\tz\ ds\\
&=:&\frac{1}{\lVert \Psi_i\rV_{L^1(\frac{3}{4}Q_i)}} (I+II).
\end{array}
\end{equation}
Let us now estimate each of the terms as follows:
\begin{description}
    \item[Estimate for $I$:] We estimate this term as follows, without loss of generality, we shall assume $Q_i$ is centered at $(0,0)$.
    \begin{equation}\label{6.8}
        \begin{array}{rcl}
            |I| & {=} & \abs{\int_0^1\iint_{\frac34 Q_i}\iprod{\nabla  w_h\lbr x_0+s(\tx-x_0),t_0+s^2(\tlt-t_0)\rbr}{\tx-x_0}\Psi_i(\tx,\tlt)\ d\tz\ ds}\\
            & \overset{\redlabel{6.8b}{a}}{\leq} & r_i \int_0^1\iint_{\frac34 Q_i}\abs{\nabla  w_h\lbr x_0+s(\tx-x_0),t_0+s^2(\tlt-t_0)\rbr}\ d\tz\ ds\\
            & \overset{\redlabel{6.8c}{b}}{\leq} & r_i |Q_i|\int_0^1\fiint_{2Q_i} \abs{\nabla  w_h(\tx,\tlt)}\ d\tz\ ds\\
            & \overset{\redlabel{6.8d}{c}}{\apprle} & r_i |Q_i| \la. 
        \end{array}
    \end{equation}
    To obtain \redref{6.8b}{a}, we made use of \descref{W13} along with the bound $|\tx- x_0| \apprle r_i$, to obtain \redref{6.8c}{b}, we enlarged the cylinder noting that $s \in [0,1]$ and finally to obtain \redref{6.8d}{c}, we made use of \descref{W4}.

    \item[Estimate for $II$:] Note that we have assumed without loss of generality that $Q_i$ is centered at $(0,0)$. Since $s \in [0,1]$, we see that $\frac{\tx + s x_0 - x_0}{s} \leq \frac34$ whenever $x_0 \in B_{\frac34r_i}(0)$ and $\tx \in B_{\frac34r_i}(0)$. Thus we can make use of \cref{given} to estimate $II$ as follows:
    \begin{equation}\label{6.9}
        \begin{array}{rcl}
            |II| & \overset{\redlabel{6.9a}{a}}{\leq} & \abs{\int_0^1\iint_{\frac34 Q_i}2s(\tlt-t_0) \iprod{[\aa(x_0+s(\tx-x_0),t_0+s^2(\tlt-t_0),\nabla w + \vec{H})]_h}{\nabla \Psi_i(\tx,\tlt)}\ d\tz\ ds}\\
            && + \abs{\int_0^1\iint_{\frac34 Q_i}2s(\tlt-t_0) \iprod{[\vec{F}^{p-2}\vec{F}]_h(x_0+s(\tx-x_0),t_0+s^2(\tlt-t_0))}{\nabla \Psi_i(\tx,\tlt)}\ d\tz\ ds}\\
            & \overset{\redlabel{6.9b}{b}}{\apprle} & \frac{\la^{2-p}r_i^2}{r_i}\int_0^1\iint_{\frac34 Q_i} \abs{[\aa(x_0+s(\tx-x_0),t_0+s^2(\tlt-t_0),\nabla w + \vec{H})]_h}\ d\tz\ ds\\
            && + \frac{\la^{2-p}r_i^2}{r_i}\int_0^1\iint_{\frac34 Q_i}\abs{[\vec{F}^{p-2}\vec{F}]_h(x_0+s(\tx-x_0),t_0+s^2(\tlt-t_0))}\ d\tz\ ds\\
            & \overset{\redlabel{6.9c}{c}}{\leq} &\frac{\la^{2-p}r_i^2}{r_i}\int_0^1\iint_{\frac34 Q_i} [|\nabla w|+|\vec{H}|]_h^{p-1}(x_0+s(\tx-x_0),t_0+s^2(\tlt-t_0))\ d\tz\ ds\\
            && + \frac{\la^{2-p}r_i^2}{r_i}\int_0^1\iint_{\frac34 Q_i}|[\vec{F}]_h|^{p-1}(x_0+s(\tx-x_0),t_0+s^2(\tlt-t_0))\ d\tz\ ds\\
            & \overset{\redlabel{6.9d}{d}}{\apprle} & r_i |Q_i| \la. 
        \end{array}
    \end{equation}
    To obtain \redref{6.9a}{a}, we made use of the weak formulation of \cref{given}, to obtain \redref{6.9b}{b}, we made use of \descref{W13} and \descref{W1}, to obtain \redref{6.9c}{c}, we made use of \cref{abounded} and finally to obtain \redref{6.9d}{d}, we made use of \descref{W4}.
\end{description}
From \descref{W12}, we note that $\frac{1}{\lVert \Psi_i\rV_{L^1(\frac{3}{4}Q_i)}}  \approx \frac{1}{|Q_i|}$ and hence combining \cref{6.8} and \cref{6.9} into \cref{6.77} proves the desired estimate.
    
    \item[Case $\frac34Q_i$ crosses the lateral boundary:] In this case, we see that $\phi_h^i = 0$, thus we have
    \begin{equation}\label{6.10}
\begin{array}{rcl}
 w_h(z_0)
&=&\frac{1}{\lVert \Psi_i\rV_{L^1(\frac{3}{4}Q_i)}} \int_0^1 \frac{d}{ds}\left(\iint_{\frac{3}{4}Q_i} w_h\lbr x_0+s(\tx-x_0),t_0+s^2(\tlt-t_0)\rbr\Psi_i(\tz) \ d\tz\right)\ ds\\
&& - \frac{1}{\lVert \Psi_i\rV_{L^1(\frac{3}{4}Q_i)}} \iint_{\frac34Q_i} w_h(\tx,\tlt) \Psi_j(\tx,\tlt)  \ d\tz\\
&=&\frac{1}{\lVert \Psi_i\rV_{L^1(\frac{3}{4}Q_i)}}\int_0^1\iint_{\frac{3}{4}Q_i}\iprod{\nabla  w_h\lbr x_0+s(\tx-x_0),t_0+s^2(\tlt-t_0)\rbr}{\tx-x_0}\Psi_i(\tz)\ d\tz\ ds\\
&&+\frac{1}{\lVert \Psi_i\rV_{L^1(\frac{3}{4}Q_i)}}\int_0^1\iint_{\frac{3}{4}Q_i}\pa_t w_h\lbr x_0+s(\tx-x_0),t_0+s^2(\tlt-t_0)\rbr\ 2s(\tlt-t_0)\Psi_i(\tz)\ d\tz\ ds\\
&& - \frac{1}{\lVert \Psi_i\rV_{L^1(\frac{3}{4}Q_i)}} \iint_{\frac34Q_i} w_h(\tz) \Psi_j(\tz)  \ d\tz\\
&=:&\frac{1}{\lVert \Psi_i\rV_{L^1(\frac{3}{4}Q_i)}} (I+II - III) .
\end{array}
\end{equation}

The terms $I$ and $II$ are estimated exactly as \cref{6.8} and \cref{6.9} respectively. In order to estimate $III$, we can apply \cref{poincare} since $w_h = 0$ outside the lateral boundaries. Thus we get
\begin{equation}
\label{6.11}
    \begin{array}{rcl}
        \abs{III} & \leq & r_i \iint_{\frac34Q_i} \abs{\frac{w_h(\tx,\tlt)}{r_i}}   \ d\tz
         \apprle  r_i \iint_{\frac34Q_i} \abs{\nabla w_h(\tx,\tlt)}   \ d\tz
         \overset{\text{\descref{W4}}}{\apprle}   r_i |Q_i| \la.
    \end{array}
\end{equation}

    \item[Case $\frac34Q_i$ crosses the initial boundary:] In this case, again we can proceed as in \cref{6.10} to get
    \begin{equation*}
\begin{array}{rcl}
 w_h(z_0)
&=&\frac{1}{\lVert \Psi_i\rV_{L^1(\frac{3}{4}Q_i)}}\int_0^1\iint_{\frac{3}{4}Q_i}\iprod{\nabla  w_h\lbr x_0+s(\tx-x_0),t_0+s^2(\tlt-t_0)\rbr}{\tx-x_0}\Psi_i(\tz)\lsb{\chi}{[0,T]}\ d\tz\ ds\\
&&+\frac{1}{\lVert \Psi_i\rV_{L^1(\frac{3}{4}Q_i)}}\int_0^1\iint_{\frac{3}{4}Q_i}\pa_t w_h\lbr x_0+s(\tx-x_0),t_0+s^2(\tlt-t_0)\rbr\ 2s(\tlt-t_0)\Psi_i(\tz)\lsb{\chi}{[0,T]}\ d\tz\ ds\\
&& - \frac{1}{\lVert \Psi_i\rV_{L^1(\frac{3}{4}Q_i)}} \iint_{\frac34Q_i} w_h(\tz) \Psi_j(\tz) \lsb{\chi}{[0,T]} \ d\tz\\
&=:&\frac{1}{\lVert \Psi_i\rV_{L^1(\frac{3}{4}Q_i)}} (I+II - III) .
\end{array}
\end{equation*}
    
    and pick up $III$ as the error term which needs to be estimated. Since we are at the initial boundary, we cannot directly apply \cref{poincare} to bound $III$ and instead we proceed analogously to \cite[Estimate (4.15) and (4.16)]{adimurthi2018gradient} to again get the same bound as \cref{6.11}.
\end{description}
Combining all the estimates completes the proof of the lemma.
\end{proof}

We now have the Lipschitz regularity of the function constructed in \cref{lipschitz_extension}, the proof of which can be found in \cite[Lemma 4.16]{adimurthi2018gradient}.
\begin{lemma}
\label{lemma3.15}
The function $\varphi_{\la,h}$ constructed in  \cref{lipschitz_extension} is $C^{0,1}(\mathbb{R}^{n}\times[0,T])$ with respect to the parabolic metric given by
\[
d_\la(z_1,z_2):=\max\{\lv x_2-x_1\rv,\sqrt{\la^{p-2}\lv t_2-t_1\rv}\}.
\]
In particular, the following bound holds for any $z_1, z_2 \in \RR^{n+1}$:
\[
\lv \varphi_{\la,h}(z_2)-\varphi_{\la,h}(z_1)\rv\le C\max\{\lv x_2-x_1\rv,\sqrt{\la^{p-2}\lv t_2-t_1\rv}\},
\]
where $C$ depends on $n,p,q,\lamot,\la,T$ and $\lV w\rV_{L^1(\Om_T)}$.
\end{lemma}

We now prove a weighted estimate which follows similarly to \cite[Theorem 3.4]{bulicek2018well}.
\begin{lemma}\label{L7}
Suppose that $\om\in A_\frac{p}{q}(\Om_T)$ where $q$ as defined in \cref{def_q} and $\lvert\na w \rvert^{q},\lvert\vec{F}\rv^{q},\lvert\vec{H}\rv^{q}\in L^\frac{p}{q}_\om(\Om_T)$ hold. Then the function $g(z)$ constructed in \eqref{maxim} satisfies
\[
\iint_{\Om_T}g^p\om\ dz\leq C\iint_{\Om_T}\left(\lvert \nabla  w\rv^{p}+\lvert \vec{F}\rvert^{p}+\lvert \vec{H}\rv^{p}\right)\om\ dz,
\]
and
\[
\iint_{\Om_T}\lvert \nabla  \varphi_{\la,h}\rv^p\om\ dz\leq C\iint_{\Om_T}g^p\om\ dz,
\]
where $C = C\lbr n,p,q,\lamot,[\om]_{\frac{p}{q}}\rbr$.
\end{lemma}
\begin{proof}
From \cref{weigted_lemma3}, we get
\[
\begin{array}{ll}
\iint_{\Om_T}g^p\om\ dz
&=\iint_{\Om_T}\mm((\lvert \nabla  u\rv^{q}+\lvert \vec{F}\rvert^{q}+\lvert\vec{H}\rv^{q})\rchi_{\Om_T})^\frac{p}{q}\om\ dz \apprle\iint_{\Om_T}\left(\lvert\na u \rvert^p+\lvert \vec{F}\rvert^p+\lvert\vec{H}\rv^p\right)\om\ dz,
\end{array}
\]
which proves the first assertion.

To obtain the second assertion, we proceed as follows:
\[
\begin{array}{rcl}
\iint_{\Om_T}\lvert\na \varphi_{\la,h}\rv^p\om\ dz
&\overset{\cref{lipschitz_extension}}{=}&\iint_{E_\la^c}\lvert\na \varphi_{\la,h}\rv^p\om\ dz+\iint_{E_\la}\lvert \nabla  w_h\rv^p\ \om\ dz \vspace{1em}\\
&\overset{\text{\cref{3.34} {and} \cref{maxim}}}{\apprle}& \iint_{E_\la^c}\la^p\om\ dz+\iint_{E_\la}g^p\ \om\ dz \vspace{1em}\\
&\overset{\text{\cref{elam}}}{\apprle}& \iint_{\Om_T}g^p\om\ dz.
\end{array}
\]
\end{proof}

\begin{lemma}\label{L9}
Suppose that $\om\in A_\frac{p}{q}(\Om_T)$ where $q$ as defined in \cref{def_q} and $\lvert\na w \rvert^{q},\lvert\vec{F}\rv^{q},\lvert\vec{H}\rv^{q}\in L^\frac{p}{q}_\om(\Om_T)$ hold. Then 
\[
\iint_{\Om_T}\lvert \nabla  \varphi_{\la,h}-\nabla  w_h\rv^p\om\ dz\apprle_{(n,p,q,\lamot)}\iint_{\{\varphi_{\la,h}\ne w_h\}}\la^p\om\ dz+\iint_{\{\varphi_{\la,h}\ne w_h\}}\lvert \nabla  w_h\rv^p\om\ dz.
\]
\end{lemma}

\begin{proof}
From \cref{lipschitz_extension}, we see that  $\{\varphi_{\la,h}\ne w_h\}\subseteq E_\la^c$. Thus we have the following sequence of estimates:
\[
\begin{array}{rcl}
\iint_{\Om_T}\lvert \nabla \varphi_{\la,h}-\nabla  w_h\rv^p\om\ dz
&=&\iint_{\{\varphi_{\la,h}\ne w_h\}}\lvert \nabla \varphi_{\la,h}-\nabla  w_h\rv^p\om\ dz\vsp\\
&\apprle& \iint_{\{\varphi_{\la,h}\ne w_h\}}\lvert \nabla  \varphi_{\la,h}\rv^p\om\ dz+\iint_{\{\varphi_{\la,h}\ne w_h\}}\lvert \nabla  w_h\rv^p\om\ dz\vsp\\
&\overset{\cref{3.34}}{\apprle}&\iint_{\{\varphi_{\la,h}\ne w_h\}}\la^p\om\ dz+\iint_{\{\varphi_{\la,h}\ne w_h\}}\lvert \nabla  w_h\rv^p\om\ dz.
\end{array}
\]
This completes the proof of the lemma.
\end{proof}

We now prove a second weighted estimate which follows similarly to \cite[Theorem 3.1]{bulicek2018well}.
\begin{lemma}\label{L6_L8}
Suppose that $\om\in A_\frac{p}{q}(\Om_T)$ where $q$ as defined in \cref{def_q} and $\lvert\na w \rvert^{q},\lvert\vec{F}\rv^{q},\lvert\vec{H}\rv^{q}\in L^\frac{p}{q}_\om(\Om_T)$ hold. Then
\begin{equation}\label{lm6.8_eqn}
\iint_{\Om_T}\lvert\pa_t\varphi_{\la,h}(\varphi_{\la,h}-w_h)\rvert\om\ dz \apprle_{(n,p,q,\lamot)}\la^p \om(\elam^c).
\end{equation}
\end{lemma}
\begin{proof}
From \cref{lipschitz_extension}, we see that if $z \in \elam$, then $\varphi_{\la,h}(z) = w_h(z)$ and hence the term on the left hand side of \cref{lm6.8_eqn} is zero. Thus, we only have to integrate the term on the left hand side of \cref{lm6.8_eqn} over $\elam^c$.

Let $z \in \elam^c$, then there exists an $i \in \NN$ such that $z \in \frac34Q_i$. Thus we get
\[
\begin{array}{rcl}
\lvert\pa_t\varphi_{\la,h}(z)(\varphi_{\la,h}(z)-w_h(z))\rvert&\overset{\cref{lipschitz_extension}}{=}&\abs{\lbr \sum_{j\in A_i}\pa_t\Psi_j(z) \phi_h^j\rbr \lbr \sum_{j\in A_i}\Psi_j(z)(w_h(z)-\phi_h^j)\rbr }\\
&\overset{\text{\descref{W14}}}{=}&\abs{\lbr \sum_{j\in A_i}\pa_t\Psi_j(z) (\phi_h^j - \phi_h^i)\rbr \lbr \sum_{j\in A_i}\Psi_j(z)(w_h(z)-\phi_h^j)\rbr }\\
&\overset{\text{\descref{W13}}}{\apprle}&\la^{p-2}\left(\sum_{j\in A_i}\frac{\lvert \phi_h^j-\phi_h^i\rv}{r_i}\right)\left(\sum_{j\in A_i}\frac{\lvert w_h(z)-\phi_h^j\rv}{r_i}\right)\vspace{1em}\\
&\overset{\text{\cref{pointwise_poincare} and \cref{diff}}}{\apprle}& \la^{p}.
\end{array}
\]
Making use of the previous pointwise bound, we get
\[
\iint_{\Om_T}\lvert\pa_t\varphi_{\la,h}(\varphi_{\la,h}-w_h)\rvert\om\ dz
=\iint_{E_\la^c}\lvert\pa_t\varphi_{\la,h}(\varphi_{\la,h}-w_h)\rvert\om\ dz\vsp
 \apprle \la^p \om(\elam^c).
\]
\end{proof}

\subsection{Proof of the compactness theory.}

We now prove the quasilinear analogue of \cite[Theorem 3.1]{bulicek2018well}.
\begin{theorem}\label{lip_seq}
Let $1 < p < \infty$ and $\max\{1,p-1\}\leq q < p$ be given. Let $\om\in A_\frac{p}{q}$ be an Muckenhoupt weight and let $\{\vec{F}^k\}_{k\ge1}$, $\{\vec{H}^k\}_{k\ge1}$ be any two sequences in $L^q(\Om_T) \cap L^{p}_{\om}(\Om_T)$. Corresponding to these vector fields, let $w^k  \in L^q(0,T;W^{1,q}_0(\Om))\cap L^p_{\om}(0,T;W^{1,p}_{\om}(\Om))$ be a very weak solution of 
\[
\begin{cases}
w^k_t-\dv\aa(z,\nabla w^k+\vec{H}^k)=\dv\lv\vec{F}^k\rv^{p-2}\vec{F}^k&\text{ in }\Om_T,\\
w^k=0&\text{ on }\pa_p\Om_T.
\end{cases}
\] 
Furthermore,  assume that there exists constant $\mathbf{C}_1$ such that
\begin{equation}\label{global_assume}
\sup_{k\ge1}\left(\lVert\vec{F}^k \rVert_{L^{q}(\Om_T)}+\lVert \vec{H}^k\rVert_{L^{q}(\Om_T)}+\lVert \na w^k\rVert_{L^{q}(\Om_T)}\right)\le \mathbf{C}_1,
\end{equation}
and
\begin{equation}\label{global_weight_assume}
\sup_{k\ge1}\left(\lVert \vec{F}^k\rV_{L^p_\om(\Om_T)}+\lVert \vec{H}^k\rV_{L^p_\om(\Om_T)}+\lVert\na w^k\rV_{L^p_\om(\Om_T)}\right)\le \mathbf{C}_1.
\end{equation}
Then for any fixed $\la>1$, there exists a sequence $\{\varphi_{\la,h}^k\}_{k\ge1}\subset L^q(0,T;W^{1,q}_0(\Om))$ satisfying following:

\begin{gather}
\lVert \nabla  \varphi_{\la,h}^k \rVert_\infty+\sup_{z_1,z_2\in\Om_T}\frac{\lvert \varphi_{\la,h}^k(z_1)-\varphi_{\la,h}^k(z_2)\rvert}{\lvert x_1-x_2\rv+\sqrt{\la^{2-p}\lvert t_1-t_2\rv}}\le C(n,p,q,\lamot,\la,T,\mathbf{C}_1),\label{conc1}\\
\iint_{\Om_T}\lvert \nabla  \varphi_{\la,h}^k\rv^p\om\ dz \le C(n,p,q,\lamot,[\om]_{\frac{p}{q}},\mathbf{C}_1),\label{conc2}\\
\iint_{\Om_T}\lvert \pa_t\varphi_{\la,h}^k(\varphi_{\la,h}^k-w_h^k)\rvert\om\ dz\apprle_{(n,p,q,\lamot,\mathbf{C}_1,[w]_{\frac{p}{q}})}\frac{1}{\la^{p-1}},\label{conc3}\\
\sup_{k\ge1}\om\left( \{z\in\Om_T\mid\varphi_{\la,h}^k(z)\ne w_h^k(z)\}\right)\apprle_{(n,p,q,\lamot,[w]_{\frac{p}{q}},\mathbf{C}_1)}\frac{1}{\la^{2p}},\label{conc4}\\
\sup_{k\ge1}\left\lv \{z\in\Om_T\mid\varphi_{\la,h}^k(z)\ne w_h^k(z)\}\right\rv\apprle_{(n,p,q,\lamot,\mathbf{C}_1)} \frac{1}{\la^{q}}.\label{conc5}
\end{gather}

\end{theorem}

\begin{proof}
For each $k\ge1$, let
\begin{equation*}
g_k(z):=\mm\left(\left[\lvert  \nabla  w^k\rv^{q}+\lvert\na\vec{F}^k\rv^{q}+\lvert\vec{H}^k\rv^{q}\right]\rchi_{\Om_T}\right)^\frac{1}{q}(z).
\end{equation*}
For each $k\ge1$ and fixed $\la>1$, we shall find a suitable $\la_k$ and by an abuse of notation, we will denote $\varphi^k_{\la,h}=\varphi^k_{\la_k,h}$. 
 Using \cref{L7} along with \cref{global_weight_assume}, we have
\begin{equation}\label{6.24}
\sup_{k\ge1}\iint_{\Om_T} g_k^p\om \ dz\le C(n,p,q,[\om]_{\frac{p}{q}},\mathbf{C}_1).
\end{equation}

Let $m_0 \in \NN$ be a number to be eventually chosen, then for each $k \geq 1$, there exists an $0 \leq m_k \leq m_0$ such that 
\begin{equation}\label{6.244}
\begin{array}{rcl}
\iint_{\{z\in \Om_T\mid \la^{2^{m_k}}<g_k(z)\le \la^{2^{m_k+1}}\}}g_k^p\om \ dz
&=&\min_{0\le m\le m_0}\iint_{\{z\in \Om_T\mid \la^{2^m}<g_k(z)\le \la^{2^{m+1}}\}}g_k^p\om \ dz\vsp\\
&\le&\frac{1}{m_0}\sum_{0\le m\le m_0}\iint_{\{z\in \Om_T\mid \la^{2^m}<g_k(z)\le \la^{2^{m+1}}\}}g_k^p\om\ dz\vsp\\
&\le& \frac{1}{m_0}\iint_{\Om_T}g_k^p\om\ dz\vsp\\
&\overset{\cref{6.24}}{\le}& \frac{C(n,p,q,[w]_{\frac{p}{q}},\mathbf{C}_1)}{m_0}.
\end{array}
\end{equation}

Let us    now define
\begin{equation*}
    \la_k := \la^{2^{m_k}} \txt{and} E_{\la}^k:=E^k_{\la_k}:=\{z\in\mathbb{R}^{n+1}\mid g_k(z)\le \la_k\}\subseteq\{z\in\mathbb{R}^{n+1}\mid\varphi_{\la_k,h}^k(z)=w_h^k(z)\},
\end{equation*}
then we have
\begin{equation}\label{6.26}
\begin{array}{rcl}
\la_k^p\om(\Om_T\setminus E_{\la}^k)
&\le& \int_{\{z\in\Om_T\mid \la_k< g_k(z)\le \la_k^2\}}g_k^p\om \ dz+\int_{\{z\in\Om_T\mid \la_k^2<g_k(z)\}}\la_k^p\om\ dz\vsp\\
&\le& \int_{\{z\in\Om_T\mid \la_k< g_k(z)\le \la_k^2\}}g_k^p\om \ dz+\frac{1}{\la_k^p}\int_{\{z\in\Om_T\mid \la_k^2<g_k(z)\}}g_k^p\om\ dz\vsp\\
&\overset{\cref{6.244}}{\apprle}&\frac{1}{m_0}+\frac{1}{\la_k^p}.\\
\end{array}
\end{equation}
\end{proof}
Let us now make the choice
\[
    m_0 = \la^p,
\]
then we have \begin{equation}\label{6.28} \la \leq \la_k := \la^{2^{m_k}} \leq \la^{2^{m_0}} =: \la^{2^{\la^p}}, \end{equation}which combined with \cref{6.26} gives 
\begin{equation}\label{imp_bnd}
    \la_k^p \om(\Om_T \setminus \elam^k) \apprle \frac{1}{\la^p}.
\end{equation}
We are now ready to prove each of the conclusions of the theorem:
\begin{description}
    \item[Proof of \cref{conc1}:] This follows from \cref{lemma3.15} since the Lipschitz extension is given by $\varphi_{\la_k,h}^k$. It is important to note that the bounds are independent of $k \in \NN$ because of \cref{global_assume} and \cref{global_weight_assume}.
    \item[Proof of \cref{conc2}:] This follows from \cref{L7} and making use of \cref{global_assume}.
    \item[Proof of \cref{conc3}:]  This follows from \cref{L6_L8} along with the bound from \cref{imp_bnd}.
    \item[Proof of \cref{conc4}:] Since $\{z\in\Om_T\mid\varphi_{\la,h}^k(z)\ne w_h^k(z)\} \subseteq \Om_T \setminus \elam^k$, we can directly use \cref{imp_bnd} and \cref{6.28} to prove this estimate.
    \item[Proof of \cref{conc5}:] Since $\{z\in\Om_T\mid\varphi_{\la,h}^k(z)\ne w_h^k(z)\} \subseteq \Om_T \setminus \elam^k$, we can apply the weak type $1-1$ estimate from \cref{max_bound} to get
    \[\begin{array}{rcl}
        \abs{\{z\in\Om_T\mid\varphi_{\la,h}^k(z)\ne w_h^k(z)\}} & \leq &  \abs{ \{g_k(z) > \la_k\}}  \\
        & \leq &  \frac{1}{\la_k^q} \iint_{\RR^{n+1}}\left[\lvert  \nabla  w^k\rv^{q}+\lvert\na\vec{F}^k\rv^{q}+\lvert\vec{H}^k\rv^{q}\right]\rchi_{\Om_T}  \ dz\\
        & \overset{\text{\cref{global_assume} and \cref{6.28}}}{\leq} & \frac{C(n,q,\mathbf{C}_1) }{\la^q}.
        \end{array}
    \]

\end{description}

\section{Proof of \texorpdfstring{\cref{main1}}.}
\label{section7}

In this section, let us take $\be = \be_0$ where $\be_0$ is from \cref{def_be_0}.  Subsequently, we shall apply the results from \cref{section6} with 
\begin{equation*}
    \begin{cases}
        q = {p-1} & \txt{if} 2 \leq p, \\
        q = 1 & \txt{if} 1<p<2.
    \end{cases}
\end{equation*}
 Let us define  \begin{equation*} \phi := (1 + |\vec{f}| + |\vec{h}|)\lsb{\chi}{\Om_T}.\end{equation*}
Let us apply \cref{weighted_estimate} with the above choice of $\phi$. Thus, we have
\begin{equation}\label{eq7.1}
     \iint_{\Om_T} |\nabla u|^p \mm(|\phi|\lsb{\chi}{\Om_T})^{-\be}(z) \ dz \apprle_{(n,p,\lamot,\Om,\om(\Om_T))} \lbr \iint_{\Om_T} \lbr |\vec{f}|^p + |\vec{h}|^p +1\rbr\mm(|\phi|\lsb{\chi}{\Om_T})^{-\be}(z) dz \rbr^{d}.
 \end{equation}
 From the properties of the Hardy-Littlewood Maximal function, we also have
 \begin{equation}\label{eq7.2}
     \lbr |\vec{f}|^p + |\vec{h}|^p +1\rbr\mm(|\phi|\lsb{\chi}{\Om_T})^{-\be}(z)  \apprle (1 + |\vec{f}| + |\vec{h}|)^{p-\be} \lsb{\chi}{\Om_T}.  
 \end{equation}
This implies that the assumption $\vec{f} \in L^{p-\be}(\Om_T)$ gives $\vec{f} \in L^p_{\om}(\Om_T)$ where we have set 
\begin{equation}\label{def_weight}
    \om(z) := \mm(|\phi|\lsb{\chi}{\Om_T})^{-\be}(z).
\end{equation}

\begin{claim}
\label{claim_weight}
    For $\be_4 := \min \left\{ \frac{1}{p-1}, p-1\right\}$, the following holds for all $\de \in (0,\be_4)$:
\[
    \begin{cases}
        \mm(|\phi|\lsb{\chi}{\Om_T})^{-\de}(z) \in A_{\frac{p}{p-1}} &\txt{if} p \geq 2, \\
        \mm(|\phi|\lsb{\chi}{\Om_T})^{-\de}(z) \in A_{p} &\txt{if} 1<p<2.
    \end{cases}\]
\end{claim}
\begin{proof}
    Let us first consider the case $p \geq 2$. From \cref{3obs_rmk}, we see that 
    \[
      \mm(|\phi|\lsb{\chi}{\Om_T})^{-\de}(z) \in A_{\frac{p}{p-1}},
    \]
    provided there exists $\al \in (0,1)$ such that the following holds
    \[
      \mm(|\phi|\lsb{\chi}{\Om_T})^{-\de}(z) = \mm(|\phi|\lsb{\chi}{\Om_T})^{\al \lbr 1-\frac{p}{p-1}\rbr}(z).
    \]
    In particular, we would need $\al = \de (p-1) < 1$, and this is possible if we choose $\de < \frac{1}{p-1}$. 
    
    In the case $1<p<2$, we proceed analogously to require would need $\al = \frac{\de}{p-1} < 1$ which again holds provided $\de < p-1$.  
\end{proof}

 A simple application of H\"older's inequality gives 
 \begin{equation*}
     \begin{array}{rcl}
      \iint_{\Om_T} |\nabla u|^{p-\beta} \ dz & = &    \iint_{\Om_T} |\nabla u|^{p-\beta} \mm(|\phi|\lsb{\chi}{\Om_T})^{\frac{\be(p-\beta)}{p}}(z) \mm(|\phi|\lsb{\chi}{\Om_T})^{\frac{(\beta-p)\be}{p}}(z)\ dz \\
      & \apprle & \iint_{\Om_T} |\nabla u|^p \mm(|\phi|\lsb{\chi}{\Om_T})^{-\be}(z) \ dz + \iint_{\Om_T} \mm(|\phi|\lsb{\chi}{\Om_T})^{p-\be}(z) \ dz \\
      & \apprle & \iint_{\Om_T} |\nabla u|^p \mm(|\phi|\lsb{\chi}{\Om_T})^{-\be}(z) \ dz + \iint_{\Om_T} (1 + |\vec{f}| + |\vec{h}|)^{p-\beta} \ dz \\
      & \overset{\text{\cref{eq7.1} and \cref{eq7.2}}}{<} & \infty.
     \end{array}
 \end{equation*}

 Since we are given $\vec{f} \in L^{p-\be}(\Om_T)$ and $\vec{h} \in L^{p-\be}(\Om_T)$, we shall consider the following approximation sequence for any $k \in (0,\infty)$:
 \begin{enumerate}[(i)]
     \item Define $\vec{f}_k:= T_k(\vec{f})$ where $T_k(s) =s$ if $|s| \leq k$ and $T_k(s) = k$ if $|s| \geq k$. Note that we have used the notation $T_k(\vec{f})$ to denote $T_k$ acting on each component of $\vec{f}$ separately.
     \item Define $\vec{h}_k=\nabla h_k$ where $h_k\in C^\infty(\Om_T)$ is from \cref{h_hyp_3} of \cref{definition_assumption}.
 \end{enumerate}
 From \cref{eq7.2} and the above construction, we see that for every $k \geq 1$, we have $\vec{f}_k\in L^\infty(\Om_T)$, $\vec{h}_k\in C^\infty(\Om_T)$ and 
\begin{equation}\label{strong_conv}
\begin{array}{ll}
\vec{f}_k\to\vec{f}&\text{ in } L^{p-\beta}(\Om_T)\cap L^p_\om(\Om_T),\\
\vec{h}_k\to\vec{h}&\text{ in } L^{p-\beta}(\Om_T)\cap L^p_\om(\Om_T).
\end{array}
\end{equation}

As a consequence, for each $k \in \NN$, there exists a unique weak solution $u^k\in C(0,T;L^{2}(\Om))\cap L^p(0,T;W_0^{1,p}(\Om))$ of 
\begin{equation}\label{k_sol}
\begin{cases}
u^k_t-\dv\aa(z,\na u^k+\vec{h}_k)=\dv|\vec{f}_k|^{p-2}\vec{f}_k&\text{ in }\Om_T,\\
u^k=0&\text{ on }\pa_p\Om_T.
\end{cases}
\end{equation}
We have the following observations:
\begin{itemize}
    \item From \cref{unweighted_estimate} applied to \cref{k_sol}, we see that 
    \begin{equation}\label{obs1}
        \|\nabla u^k\|_{L^{p-\be}(\Om_T)} \apprle \| \vec{h}_k + \vec{f}_k\|_{L^{p-\be}(\Om_T)} \overset{\cref{strong_conv}}{\apprle} \| \vec{h} + \vec{f}\|_{L^{p-\be}(\Om_T)} < \infty.
    \end{equation}
    \item From \cref{weighted_estimate} applied to \cref{k_sol} with weight defined as in \cref{def_weight}, we have
     \begin{equation*}
     \begin{array}{rcl}
     \iint_{\Om_T} |\nabla u^k|^p \om(z) \ dz & \apprle &  \lbr \iint_{\Om_T} \lbr |\vec{f}_k|^p + |\vec{h}_k|^p +1\rbr\om(z) dz \rbr^{d} \\
     & \overset{\cref{eq7.2}}{\apprle} &\lbr \iint_{\Om_T} \lbr |\vec{f}_k|^{p-\be} + |\vec{h}_k|^{p-\be} +1\rbr  dz \rbr^{d} \overset{\cref{strong_conv}}{<} \infty.
     \end{array}
    \end{equation*}
    \item Furthermore, applying the above two observations to \cref{k_sol}, we see that
    \begin{equation*}
        \|u^k_t\|_{L^\frac{p-\beta}{p-1}(0,T;W^{-1,\frac{p-\beta}{p-1}}(\Om))}  \apprle  \|\nabla u^k\|_{L^{p-\be}(\Om_T)} + \|\vec{f}_k\|_{L^{p-\be}(\Om_T)} + \|\vec{h}_k\|_{L^{p-\be}(\Om_T)}
         \overset{\cref{obs1}}{<}  \infty.
    \end{equation*}
\end{itemize}
As a consequence of the above observations, we have the following convergences (upto relabelling a suitable subsequence):
\begin{equation}\label{convergence}
\begin{array}{rcll}
u^k&\to& u&\text{ strongly in }L^{p-\beta}(\Om_T)\cap L^2(\Om_T),\\
\nabla u^k&\rightharpoonup& \nabla u &\text{ weakly in }L^{p-\beta}(\Om_T)\cap L^p_\om(\Om_T),\\
\aa^k:=\aa(z,\nabla u^k+\vec{h}_k)&\rightharpoonup& \bar{\aa}&\text{ weakly in }L^\frac{p-\beta}{p-1}(\Om_T)\cap L^\frac{p}{p-1}_\om(\Om_T).
\end{array}
\end{equation}
To get the above convergence results, we made use of \cref{Aubin_Lions} along with the restriction \cref{restriction_p} which ensures that the Sobolev exponent $p^*:= \frac{np}{n-p} > 2$.

From \cref{convergence}, we see that the function $u \in L^{p-\be}(0,T;W_0^{1,p-\be}(\Om))$ is a distributional solution of 
\[
\begin{cases}
u_t-\dv\bar{\aa}=\dv|\vec{f}|^{p-2}\vec{f}&\text{ in }\Om_T,\\
u=0&\text{ on }\pa_t\Om_T,
\end{cases}
\]
where the operator $\bar{\aa}$ is a formal limit as obtained in \cref{convergence}.

 Denote $v^k:=u^k-u$, then $v^k$ is a  weak solution of 
\begin{equation}\label{diff_eq}
\begin{cases}
v^k_t-\dv\aa(z,\na v^k+\vec{h}_k-\na u)=\dv\bar{\aa}+\dv(\lv \vf_k\rv^{p-2}\vf_k-\lv\vf\rv^{p-2}\vf)&\text{ in }\Om_T,\\
v^k=0&\text{ in }\pa_p\Om_T.
\end{cases}
\end{equation}

From \cref{convergence}, we see that all the hypothesis of \cref{lip_seq} are satisfied for the weak solutions of \cref{diff_eq}. Hence for each fixed $\la>1$, there exists a a family of functions $\{ v_{\la,h}^k\}_{k\geq 1}$ satisfying the conclusions of \cref{lip_seq}.

\begin{claim}
\label{claim7.2}
    Suppose \cref{claim7.1} holds, then the following holds:
    \begin{equation}
        \label{minty}
        \bar{\aa} = \aa(z,\nabla u + \vec{h}).
    \end{equation}

\end{claim}

\begin{proof}
    From \cref{convergence}, \cref{diff_eq} and \cref{lip_seq}, we see that $\{\iprod{\aa^k}{\na u^k-\na u}\om\}_{k\ge1}$, $\{\lv\na u^k-\na u\rv^p\om\}_{k\ge1}$, $\{ |\nabla v^k| \om \}_{k \geq 1}$ and $\{ |\nabla v_{\la,h}^k| \om \}_{k \geq 1}$ are bounded sequences in $L^1(\Om_T)$. As a consequence, we can apply \cref{l1_compact} to get a sequence of measurable sets $\{F_j\}_{\{j\ge1\}}$ such that $F_j\subset\Om_T$, $\lv\Om_T\setminus F_j\rv\to0$ as $j\to\infty$ and having the property that $\{\iprod{\aa^k}{\na u^k-\na u}\om\}_{k\ge1}$ and $\{\lv\na u^k-\na u\rv^p\om\}_{k\ge1}$ are precompact in $L^1(F_j)$ for all $j\ge1$.
    
    Therefore taking subsequence if necessary, for any $\ve>0$ and $j\ge1$, there exists a $\delta(\ep,j)>0$ such that for any $\mathcal{O}\subset F_j$, there holds
    \begin{equation*}
\lv\mathcal{O}\rv<\delta\qquad \Rightarrow \qquad\sup_{k\ge1}\iint_{\mathcal{O}}\lv \na u^k-\na u\rv^p\om\ dz<\ep.
\end{equation*}

From \cref{claim7.1}, we see that the following holds with $F = F_j$ for any $j \in \NN$:
\[
\lim_{k\to\infty}\iint_{F_j}\iprod{\aa^k}{\na u^k}\om\ dz=\lim_{k\to\infty}\iint_{F_j}\iprod{\aa^k}{\na u^k-\na u}\om\ dz +\lim_{k\to\infty}\iint_{F_j}\iprod{\aa^k}{\na u}\om\ dz=\lim_{k\to\infty}\iint_{F_j}\iprod{\bar{\aa}}{\na u}\om\ dz.
\]

Therefore for any $\Theta\in L^{p}_\om(\Om_T)$, the following sequence of estimates hold:
\begin{equation}\label{725}
\begin{array}{rcl}
0
&\overset{\eqref{abounded}}{\le}& \lim_{k\to\infty}\iint_{F_j}\iprod{\aa^k-\aa(z,\Th+\vec{h})}{\na u^k-\Th+\vec{h}_k-\vec{h}}\om\ dz\vsp\\
&=&\lim_{k\to\infty}\iint_{F_j}\iprod{\aa^k-\aa(z,\Th+\vec{h})}{\na u^k-\Th}\ dz+\lim_{k\to\infty}\iint_{F_j}\iprod{\aa^k-\aa(z,\Th+\vec{h})}{\vec{h}^k-\vec{h}}\om\ dz\vsp\\
&\overset{\cref{convergence}}{=}&\lim_{k\to\infty}\iint_{F_j}\iprod{\aa^k-\aa(z,\Th+\vec{h})}{\na u^k-\Th}\om\ dz\vsp\\
&\overset{\cref{convergence}}{=}&\iint_{F_j}\iprod{\bar{\aa}-\aa(z,\Th+\vec{h})}{\na u-\Th}\om\ dz.
\end{array}
\end{equation}

Let us take  $\Th = \na u+\ve\left(\frac{\bar{\aa}-\aa(z,\na u+\vec{h})}{1+\lv\bar{\aa}-\aa(z,\na u+\vec{h})\rv}\right)$ where second term belong to $L^\infty(\Om_T)$, we get
\[
\begin{array}{rcl}
0&\overset{\cref{725}}{\le}& \lim_{\ve\to0^+}\int_{F_j}\iprod{\bar{\aa}-\aa\left(z,\na u+\vec{h}+\ve\left(\frac{\bar{\aa}-\aa(z,\na u+\vec{h})}{1+\lv\bar{\aa}-\aa(z,\na u+\vec{h})\rv}\right)\right)}{-\left(\frac{\bar{\aa}-\aa(z,\na u+\vec{h})}{1+\lv\bar{\aa}-\aa(z,\na u+\vec{h})\rv}\right)}\om\ dz\vsp\\
&=&\int_{F_j}\frac{-\left\lv\bar{\aa}-\aa(z,\na u+\vec{h})\right\rv^2}{1+\lv\bar{\aa}-\aa(z,\na u+\vec{h})\rv} \om\ dz.
\end{array}
\]
Since $\om\ge0$, we now let  $j\to\infty$ to obtain we obtain \eqref{minty}.
\end{proof}

We are now left to prove the following claim:
\begin{claim}
    \label{claim7.1}
    Let $j \in \NN$ and $F_j$ be a measurable set obtained in \cref{claim7.2}. We then claim that the following limit holds:
    \begin{equation*}
\iprod{\aa(z,\nabla u^k+\vec{h}_k)}{\na u^k-\na u}\om\rightharpoonup  0\text{ weakly in }L^1(F_j).
\end{equation*}

\end{claim}
\begin{proof}

From \cref{lipschitz_extension}, \cref{item4} and  \cref{convergence}, we see that 
\[
\iint_{\Om_T}\lv \vlh^k\rv^{p-\beta}\ dz\apprle\iint_{\Om_T}\lv v_h^k\rv^{p-\beta}\ dz\to0\text{ as }k\to\infty.
\]
Therefore by \cref{conc1} and Arzel\`{a}-Ascoli Theorem, taking a subsequence if necessary, we get
\begin{equation}\label{conv1}
\vlh^k\to0\text{ in }L^\infty(\Om_T)\text{ as }k\to\infty.
\end{equation}

Also, since for any $\eta\in C_0^\infty(\Om_T)$, there holds
\[
\lim_{k\to\infty}\iint_{\Om_T}\eta\na \vlh^k\ dz=\lim_{k\to\infty}\iint_{\Om_T}-\vlh^k\na\eta\ dz=0,
\]
from which we have
\begin{equation}\label{conv2}
\na\vlh^k\overset{\ast}{\rightharpoonup}0\qquad \text{  in weak $*$ }L^\infty(\Om_T,\mathbb{R}^{n+1})\text{ as }k\to\infty.
\end{equation}

Making use of  \cref{conv1}, for any $\eta\in C^\infty(\Om_T)$ with $\eta = 0$ on $\pa_p(\Om_T)$, we have
\begin{equation}\label{seq0}
\begin{array}{rcl}
\lim_{k\to\infty}\iint_{\Om_T}\iprod{[\aa^k-\bar{\aa}]_h}{\na \vlh^k}\eta\ dz & = & \lim_{k\to\infty}\iint_{\Om_T}\iprod{[\aa^k-\bar{\aa}]_h}{\na(\eta\vlh^k)}\ dz-\lim\iint_{\Om_T}\iprod{[\aa^k-\bar{\aa}]_h}{\na\eta}\vlh^k\ dz\vsp\\
&=&\lim_{k\to\infty}\iint_{\Om_T}\iprod{[\aa^k-\bar{\aa}]_h}{\na(\eta\vlh^k)}\ dz.
\end{array}
\end{equation}

Making use of \cref{strong_conv} along with \cref{diff_eq}, \cref{conv1} and  \cref{conv2}, we see that \cref{seq0} becomes
\begin{equation*}
\begin{array}{ll}
&\hspace*{-2cm}\lim_{k\to\infty}\iint_{\Om_T}\iprod{[\aa^k-\bar{\aa}]_h}{\na \vlh^k}\eta\ dz\vsp\\
&=-\lim_{k\to\infty}\iint_{\Om_T}\pa_tv_h^k(\eta\vlh^k)\ dz-\lim_{k\to\infty}\iint_{\Om_T}\iprod{[\lv \vf_k\rv^{p-2}\vf_k-\lv \vf\rv^{p-2}\vf]_h}{\na(\eta\vlh^k)}\ dz\vsp\\
&=-\lim_{k\to\infty}\iint_{\Om_T}\pa_tv_h^k(\eta\vlh^k)\ dz.
\end{array}
\end{equation*}

A simple calculation shows
\begin{equation}\label{7.19}
\begin{array}{rcl}
\iint_{\Om_T}\pa_tv_h^k(\eta\vlh^k)\ dz &=&\frac{1}{2}\iint_{\Om_T}\pa_t\left(\left[(v_h^k)^2-(\vlh^k-v_h^k)^2\right]\eta\right)\ dz-\frac{1}{2}\iint_{\Om_T}\pa_t\eta\left[(v_h^k)^2-(\vlh^k-v_h^k)^2\right]\ dz\vsp\\
&&+\iint_{\Om_T}\pa_t\vlh^k(\vlh^k-v_h^k)\eta\ dz,
\end{array}
\end{equation}
Using \cref{convergence} along with \cref{conv1}, we see that the first two terms on the right hand side of \cref{7.19} goes to zero as $k \rightarrow \infty$. Thus, \cref{seq0} becomes
\[
\lim_{k\to\infty}\iint_{\Om_T}\iprod{[\aa^k-\bar{\aa}]_h}{\na \vlh^k}\eta\ dz=\lim_{k\to\infty}\iint_{\Om_T}\pa_t\vlh^k(\vlh^k-v_h^k)\eta\ dz.
\]

From  \cref{convergence} and \cref{conc1}, we see that 
\[
\sup_{k\ge1}\left\lV\iprod{[\aa^k-\bar{\aa}]_h}{\na \vlh^k}\right\rV_{L^\frac{p-\beta}{p-1}(\Om_T)}<\infty,
\]
from which (after possibly taking a subsequence if necessary), for any $\eta\in L^{\left(\frac{p-\beta}{p-1}\right)'}(\Om_T)$, we have
\begin{equation}\label{seq2}
\lim_{k\to\infty}\iint_{\Om_T}\iprod{[\aa^k-\bar{\aa}]_h}{\na \vlh^k}\eta\ dz=\lim_{k\to\infty}\iint_{\Om_T}\pa_t\vlh^k(\vlh^k-v_h^k)\eta\ dz.
\end{equation}

Since our weight function from \cref{def_weight} $\om\in L^\infty(\Om_T)$, we take $\eta=\om\rchi_{F_j}$ where $F_j$ is some measurable subset of $\Om_T$. Now applying H\"{o}lder's inequality and \cref{conc3} to the expression on the right hand side of \eqref{seq2}, we get
\begin{equation}\label{seq22}
\lim_{k\to\infty}\left\lv\iint_{F_j}\iprod{[\aa^k-\bar{\aa}]_h}{\na \vlh^k}\om\rchi_{F_j}\ dz\right\rv\apprle\frac{1}{\la^{p-1}}.
\end{equation}
And since $\rchi_{F_j}\om\bar{\aa}\in L^1(\Om_T)$, from \eqref{conv2}, we see that $[\bar{\aa}]_h \nabla v_{\la,h}^k \om\rchi_{F_j} \rightarrow 0$. Thus we have
\begin{equation}\label{seq3}
\lim_{k\to\infty}\left\lv\iint_{F_j}\iprod{[\aa^k]_h}{\na \vlh^k}\om\rchi_{F_j}\ dz\right\rv\apprle\frac{1}{\la^{p-1}}.
\end{equation}

Making use of \cref{seq3} in \cref{seq22}, we have the following sequence of estimates:
\[
\begin{array}{rcl}
\lim_{k\to\infty}\left\lv\iint_{F_j}\iprod{[\aa^k]_h}{\na v_h^k}\om\ dz\right\rv\vsp
&\le &\lim_{k\to\infty}\left\lv\iint_{F_j}\iprod{[\aa^k]_h}{\na \vlh^k}\om\ dz\right\rv+\lim_{k\to\infty}\left\lv\iint_{F_j}\iprod{[\aa^k]_h}{\na(\vlh^k- v_h^k)}\om\ dz\right\rv\vsp\\
&\overset{\cref{seq3}}{\apprle}& \frac{1}{\la^{p-1}}+\left( \iint_{F_j}\lv[\aa^k]_h\rv^\frac{p}{p-1}\om\ dz\right)^\frac{p-1}{p}\left(\iint_{F_j} \lv\na(\vlh^k-v_h^k)\rv^p\om\ dz\right)^\frac{1}{p}\vsp\\
&\overset{\cref{convergence}}{\apprle}& \frac{1}{\la^{p-1}}+\left(\iint_{F_j} \lv\na(\vlh^k-v_h^k)\rv^p\om\ dz\right)^\frac{1}{p}\vsp\\
& \apprle & \frac{1}{\la^{p-1}}+\left(\iint_{F_j\cap\{\vlh^k\ne v_h^k\}}\lv \na\vlh^k\rv^p\om\ dz\right)^\frac{1}{p}+\left(\iint_{F_j\cap\{\vlh^k\ne v_h^k\}}\lv \na v_h^k\rv^p\om\ dz\right)^\frac{1}{p}\vsp\\
&\overset{\redlabel{aa}{a}}{\apprle}& \frac{1}{\la^{p-1}} + o\lbr\frac{1}{\la}\rbr,
\end{array}
\]
where to obtain \redref{aa}{a}, we made use of \cref{convergence} and \cref{conc2} along with  \cref{conc5}. Here $o\lbr\frac{1}{\la}\rbr$ denotes a quantity that goes to zero as $\la \rightarrow \infty$ which holds due to \cref{lip_seq} and the choice of measurable set $F_j$. We now let $\la \rightarrow \infty$ which proves the claim.
\end{proof}

This completes the proof of the existence result.


\section*{References}   

\begin{thebibliography}{10}

\bibitem{acerbifusco}
Emilio Acerbi and Nicola Fusco.
\newblock An approximation lemma for {$W^{1,p}$} functions.
\newblock In {\em Material instabilities in continuum mechanics ({E}dinburgh,
  1985--1986)}, Oxford Sci. Publ., pages 1--5. Oxford Univ. Press, New York,
  1988.

\bibitem{MR2286632}
Emilio Acerbi and Giuseppe Mingione.
\newblock Gradient estimates for a class of parabolic systems.
\newblock {\em Duke Math. J.}, 136(2):285--320, 2007.

\bibitem{AH}
David~R Adams and Lars~I Hedberg.
\newblock {\em Function {S}paces and {P}otential {T}heory}, volume 314.
\newblock Grundlehren der Mathematischen Wissenschaften, Springer, Berlin,
  1996.

\bibitem{adimurthi2018gradient}
Karthik Adimurthi and Sun-Sig Byun.
\newblock {Gradient weighted estimates at the natural exponent for Quasilinear
  Parabolic equations}.
\newblock {\em arXiv preprint
  \href{https://arxiv.org/abs/1804.04356}{arXiv:1804.04356}}, 2018.
  
\bibitem{tadele}
Karthik Adimurthi, Tadele Mengesha and Nguyen~Cong Phuc.
\newblock {Gradient weighted norm inequalities for linear elliptic equations with discontinuous coefficients}.
\newblock To appear, {\em Appl. Math and Optimization}.
\newblock {\href{https://arxiv/org/abs/1806.00423}{https://arxiv.org/abs/1806.00423}}, 2016.

\bibitem{adimurthi2016weight}
Karthik Adimurthi and Nguyen~Cong Phuc.
\newblock An end-point global gradient weighted estimate for quasilinear
  equations in non-smooth domains.
\newblock {\em Manuscripta Math.}, 150(1-2):111--135, 2016.

\bibitem{AP2}
Karthik Adimurthi and Nguyen~Cong Phuc.
\newblock An end-point global gradient weighted estimate for quasilinear
  equations in non-smooth domains.
\newblock {\em Manuscripta Mathematica}, 150(1-2):111--135, 2016.

\bibitem{ball1989remarks}
J.~M. Ball and F.~Murat.
\newblock Remarks on {C}hacon's biting lemma.
\newblock {\em Proc. Amer. Math. Soc.}, 107(3):655--663, 1989.

\bibitem{bogelein2013regularity}
Verena B\"{o}gelein, Frank Duzaar, and Giuseppe Mingione.
\newblock The regularity of general parabolic systems with degenerate
  diffusion.
\newblock {\em Mem. Amer. Math. Soc.}, 221(1041):vi+143, 2013.

\bibitem{Bui1}
The~Anh Bui and Xuan~Thinh Duong.
\newblock Global {L}orentz estimates for nonlinear parabolic equations on
  nonsmooth domains.
\newblock {\em arXiv preprint
  \href{https://arxiv.org/abs/1702.06202}{https://arxiv.org/abs/1702.06202}}, 2017.

\bibitem{bulicek2018well}
Miroslav Bul{\'\i}{\v{c}}ek, Jan Burczak, and Sebastian Schwarzacher.
\newblock Well posedness of nonlinear parabolic systems beyond duality.
\newblock {\em arXiv preprint
  \href{https://arxiv.org/abs/1810.05061}{https://arxiv.org/abs/1810.05061}}, 2018.

\bibitem{bulivcek2016linearexistence}
Miroslav Bul\'{i}\v{c}ek, Lars Diening, and Sebastian Schwarzacher.
\newblock Existence, uniqueness and optimal regularity results for very weak
  solutions to nonlinear elliptic systems.
\newblock {\em Anal. PDE}, 9(5):1115--1151, 2016.

\bibitem{bulivcek2016existence}
Miroslav Bul\'{i}\v{c}ek and Sebastian Schwarzacher.
\newblock Existence of very weak solutions to elliptic systems of
  {$p$}-{L}aplacian type.
\newblock {\em Calc. Var. Partial Differential Equations}, 55(3):Art. 52, 14,
  2016.

\bibitem{MR3461425}
Sun-Sig Byun, Jihoon Ok, Dian~K. Palagachev, and Lubomira~G. Softova.
\newblock Parabolic systems with measurable coefficients in weighted {O}rlicz
  spaces.
\newblock {\em Commun. Contemp. Math.}, 18(2):1550018, 19, 2016.

\bibitem{BOS1}
Sun-Sig Byun, Jihoon Ok, and Seungjin Ryu.
\newblock Global gradient estimates for general nonlinear parabolic equations
  in nonsmooth domains.
\newblock {\em Journal of Differential Equations}, 254(11):4290--4326, 2013.

\bibitem{BO1}
Sun-Sig Byun, Jihoon Ok, and Seungjin Ryu.
\newblock Global gradient estimates for elliptic equations of
  {$p(x)$}-laplacian type with {BMO} nonlinearity.
\newblock {\em Journal f{\"u}r die reine und angewandte Mathematik (Crelles
  Journal)}, 2016(715):1--38, 2016.

\bibitem{byun2013weighted}
Sun-Sig Byun, Dian~K Palagachev, and Seungjin Ryu.
\newblock {Weighted $W^{1,p}$ estimates for solutions of non-linear parabolic
  equations over non-smooth domains}.
\newblock {\em Bulletin of the London Mathematical Society}, 45(4):765--778,
  2013.

\bibitem{byun2017weighted}
Sun-Sig Byun and Seungjin Ryu.
\newblock Weighted orlicz estimates for general nonlinear parabolic equations
  over nonsmooth domains.
\newblock {\em Journal of Functional Analysis}, 272(10):4103--4121, 2017.

\bibitem{BW-CPAM}
Sun-Sig Byun and Lihe Wang.
\newblock Elliptic equations with {BMO} coefficients in {R}eifenberg domains.
\newblock {\em Communications on pure and applied mathematics},
  57(10):1283--1310, 2004.

\bibitem{DiB1}
Emmanuele DiBenedetto.
\newblock {\em Degenerate {P}arabolic {E}quations}.
\newblock Universitext, Springer-Verlag, New York, 1993.

\bibitem{diening2010existence}
Lars Diening, Michael R\u{u}\v{z}i\v{c}ka, and J\"{o}rg Wolf.
\newblock Existence of weak solutions for unsteady motions of generalized
  {N}ewtonian fluids.
\newblock {\em Ann. Sc. Norm. Super. Pisa Cl. Sci. (5)}, 9(1):1--46, 2010.

\bibitem{MR3318165}
Hongjie Dong and Doyoon Kim.
\newblock On {$L^p$}-estimates for elliptic and parabolic equations with {$A_p$} weights.
\newblock{\em Trans. Amer. Math. Soc.},  370(7):5081--5130,  2018.

\bibitem{Grafakos}
Loukas Grafakos.
\newblock {\em Classical and modern {F}ourier analysis}.
\newblock Pearson Education, Inc., Upper Saddle River, NJ, 2004.

\bibitem{KL}
Juha Kinnunen and John~L Lewis.
\newblock Very weak solutions of parabolic systems of {$p$}-laplacian type.
\newblock {\em Arkiv f{\"o}r Matematik}, 40(1):105--132, 2002.

\bibitem{johnlewis}
John~L. Lewis.
\newblock On very weak solutions of certain elliptic systems.
\newblock {\em Comm. Partial Differential Equations}, 18(9-10):1515--1537,
  1993.

\bibitem{liebermanbook}
Gary~M. Lieberman.
\newblock {\em Second order parabolic differential equations}.
\newblock World Scientific Publishing Co., Inc., River Edge, NJ, 1996.

\bibitem{Mikkonen}
Pasi Mikkonen.
\newblock On the {W}olff potential and quasilinear elliptic equations involving
  measures.
\newblock {\em Ann. Acad. Sci. Fenn. Math. Diss.}, (104):71, 1996.

\bibitem{MR2836359}
Dian~K. Palagachev and Lubomira~G. Softova.
\newblock Quasilinear divergence form parabolic equations in {R}eifenberg flat
  domains.
\newblock {\em Discrete Contin. Dyn. Syst.}, 31(4):1397--1410, 2011.

\bibitem{simons1987compact}
Jacques Simon.
\newblock Compact sets in the space {$L^p(0,T;B)$}.
\newblock {\em Ann. Mat. Pura Appl. (4)}, 146:65--96, 1987.

\bibitem{torchinsky1986real}
Alberto Torchinsky.
\newblock {\em Real-variable methods in harmonic analysis}, volume 123 of {\em
  Pure and Applied Mathematics}.
\newblock Academic Press, Inc., Orlando, FL, 1986.

\end{thebibliography}
\end{document}